\documentclass[12pt]{amsart}

\usepackage{mathtools}

\usepackage{amssymb}
\usepackage{hyperref, xcolor}
\usepackage[left=.8in,right=.8in,top=1in,bottom=1in,includefoot]{geometry}
\usepackage{paralist}

\newtheorem{theorem}{Theorem}[section]
\newtheorem*{theorem*}{Theorem}
\newtheorem{lemma}[theorem]{Lemma}
\newtheorem{proposition}[theorem]{Proposition}
\newtheorem{corollary}[theorem]{Corollary}
\newtheorem*{corollary*}{Corollary}
\newtheorem{conjecture}[theorem]{Conjecture}

\theoremstyle{definition}
\newtheorem{definition}[theorem]{Definition}
\newtheorem*{definition*}{Definition}
\newtheorem{example}[theorem]{Example}

\newcommand{\F}{\mathbf{F}}
\newcommand{\Fq}{\F_q}

\newcommand{\K}{Krawtchouk }
\newcommand{\defeq}{\overset{\mathrm{def}}{=}}

\newcommand{\GL}{\mathrm{GL}}

\newcommand{\NN}{\mathbb{N}}

\newcommand{\rk}{\operatorname{rank}}

\newcommand{\mat}{\mathfrak{m}}

\newcommand{\nHit}{H}
\newcommand{\Hit}{H}
\newcommand{\NE}{\mathrm{NE}}
\newcommand{\GR}{\mathrm{GR}}
\newcommand{\Sym}{\mathfrak{S}}

\theoremstyle{remark}
\newtheorem{remark}[theorem]{Remark}
\newtheorem{question}[theorem]{Question}

\numberwithin{equation}{section}
\newcommand\qbin[3]{\left[\begin{matrix} #1 \\ #2 \end{matrix} \right]_{#3}}

\DeclareMathOperator{\stat}{stat}
\DeclareMathOperator{\inv}{inv}

\begin{document}
\title{Rook theory of the finite general linear group}
\author{Joel Brewster Lewis}
\address{Department of Mathematics \\
         George Washington University \\
         Washington, DC 20052} 
         \email{jblewis@gwu.edu} 
\author{Alejandro H. Morales}
\address{Department of Mathematics and Statistics\\
University of Massachusetts, Amherst\\
 MA 01003} 
 \email{ahmorales@math.umass.edu}

\begin{abstract}
Matrices over a finite field having fixed rank and restricted support
are a natural $q$-analogue of rook placements on a board.  We develop
this $q$-rook theory by defining a corresponding analogue of the hit
numbers.  Using tools from coding theory, we show that
these $q$-hit and $q$-rook numbers obey a variety of identities
analogous to the classical case.  We also explore connections to
earlier $q$-analogues of rook theory, as well as settling a
polynomiality conjecture and finding a counterexample of a positivity
conjecture of the authors and Klein.
\end{abstract}

\maketitle

\section{Introduction}

Classically, part of rook theory goes like this
\cite{RiordanKaplansky}: given a \emph{board} $B$ contained in the
discrete $n \times n$ square grid $[n] \times [n]$, one wishes to find
the \emph{rook number} $r_i(B)$, the number of ways of placing $i$
non-attacking rooks in $B$, or the {\em hit number} $h_i(B)$, the
number of $n \times n$ permutation matrices with $i$ rooks in $B$.  These numbers are difficult to compute in general \cite{Valiant}, but nevertheless one can say many things about their properties.   For any board $B$, the rook and hit numbers are related by the equation
\begin{equation} 
\label{eq:classic_rookhit_rel}
\sum_{i=0}^n h_i(B) \cdot t^i = \sum_{i=0}^n r_i(B) \cdot  (n-i)! \cdot (t-1)^i.
\end{equation}
Moreover, from their definition the hit numbers satisfy the reciprocity relation
\begin{equation}
\label{eq:classic-hit-recip}
h_{n-i}(\overline{B})  = h_i(B),
\end{equation}
where $\overline{B}$ denotes the complement of $B$ with respect to $[n] \times [n]$.
The zero hit number $h_0(B) = h_n(\overline{B}) = r_n(\overline{B})$ 
is of particular interest; setting $t=0$ in \eqref{eq:classic_rookhit_rel} gives the 
inclusion-exclusion formula
\[
h_0(B) =  \sum_{i=0}^n (-1)^i \cdot (n-i)! \cdot r_i(B). 
\]
For example, this relation can be used to find formulas for the number $d_n$ of 
permutations of size $n$ with no fixed points (derangements) and the number $c_n$ of
permutations $w$ of size $n$ such that $w(i) \not\equiv i,i+1 \pmod{n}$ (the
famous {\em probl\`eme des m\'enages}; see \cite[\S 2.3]{EC1}). The
boards in these cases are the diagonal $\{(1,1),\ldots,(n,n)\}$
and a board consisting of the
diagonal, the next upper diagonal and the cell $(n,1)$ (Figure~\ref{fig:meange4by4board}).

Garsia and Remmel \cite{GarsiaRemmel} started the study of
$q$-analogues of rook numbers by defining {\em $q$-rook numbers} and
{\em $q$-hit numbers} for Ferrers boards.  By definition, these $q$-analogues are polynomials in a formal variable $q$ having nonnegative integer coefficients, whose values at $q = 1$ are equal to the corresponding rook numbers and hit numbers.  
A different kind of $q$-analogue of rook numbers was
proposed in \cite{LLMPSZ}, namely, the number of $n\times n$ 
matrices with entries in the finite field $\Fq$ with $q$ elements 
having rank $i$ and support in $B$. This number, denoted by $\mat_i(B,q)$, is an
enumerative $q$-analogue of $r_i(B)$
in a sense made precise in \eqref{eq:Mq-analogue} below%
. 
When $B$ is a Ferrers board, Haglund \cite{Haglund} had already
shown that $\mat_i(B,q)$  is equivalent to the Garsia--Remmel
$q$-rook numbers. However, for general boards, the function $\mat_i(B,q)$
need not be a polynomial in $q$ \cite{Stembridge} (indeed it can be much more complicated \cite{KaplanLM}), and if it is a polynomial it might or not have nonnegative integer coefficients.

In the first part of this paper, we continue the study of this new
$q$-rook theory. We define a corresponding notion of $q$-hit numbers for an arbitrary board $B$ using a suggestion of Remmel (private communication), 
and we give a reciprocity
relation for $\mat_i(B,q)$ and $q$-hit numbers using a result of Delsarte \cite{Delsarte} that is an analogue of the MacWilliams identity \cite{MacWilliams} on the weights of dual codes. (The connection between this identity and $\mat_i(B,q)$ has appeared in work of Ravagnani \cite[Rem.~50]{Ravagnani}.)

Since $\mat_i(B,q)$ is always divisible (as an integer) by $(q - 1)^i$, it is convenient to define the \emph{reduced} (or \emph{projective}) \emph{matrix count} $M_i(B,q)= \mat_i(B,q)/(q-1)^i$.  Then the $q$-hit numbers for an arbitrary board are defined as follows.
\begin{definition*} 
Given a board $B\subseteq[n]\times [n]$ and a nonnegative integer $i$, define the \emph{$q$-hit number} $\nHit_i(B,q)$ by the equation
\begin{equation} \label{eq:intro-hit}
\sum_{i=0}^n \nHit_i(B,q) \cdot t^i = q^{\binom{n}{2}} \sum_{i=0}^n M_i(B,q)\cdot [n-i]!_q \prod_{k=0}^{i-1} (tq^{-k}-1),
\end{equation}
where $[n - i]!_q$ is a $q$-factorial.  
Let ${P}(B,q,t)$ denote the expression on both sides of this equality.
\end{definition*}
Some properties of the hit numbers are immediate formal consequences of this definition.  By taking leading coefficients we have that $\nHit_n(B, q)$ is equal to $M_n(B, q)$, while by setting $t = 1$ we have that the $q$-hit numbers partition $|\GL_n(\F_q)|$, in the sense that
\[
(q-1)^n\sum_{i=0}^n \nHit_i(B,q) = |\GL_n(\F_q)|.
\]

Other properties are less obvious.  We show that the functions $\nHit_i(B,q)$ are enumerative $q$-analogues of the
hit numbers (Proposition~\ref{prop:qhitqanalogue}), that they
coincide with the Garsia--Remmel $q$-hit numbers when $B$ is a Ferrers
board (Proposition~\ref{prop:q-hit and gr}), and that their generating
function ${P}(B, q, t)$ has a probabilistic interpretation
(Theorem~\ref{thm:probinterpret}).  Furthermore, using a generalized MacWilliams complement identity for $M_i(B,q)$, we show in Section~\ref{sec:reciprocity} the following reciprocity of $q$-hit numbers.
\begin{theorem*}
For every board $B\subseteq [n]\times [n]$ and for $i=0,\ldots,n$ we
have that 
\[
\nHit_{n-i}(\overline{B},q) = q^{in-|B|} \cdot \nHit_i(B,q).
\]
\end{theorem*}

\noindent
We leave open the problem of giving a combinatorial interpretation to
$\nHit_{i}(B,q)$ (Question~\ref{open question:interpretation}). 

As in the classical case, the zeroth $q$-hit number is particularly nice.  
By the $q$-hit reciprocity, one can show that $\nHit_0(B,q)  = q^{|B|}M_n(\overline{B},q)$.
Moreover, there is an inclusion-exclusion formula for this number (Corollary~\ref{cor:H_0 as inclusion-exclusion}). 

\begin{corollary*}
For any board $B \subset [n] \times [n]$ we have
\[
\nHit_0(B,q)  = q^{|B|} M_n(\overline{B},q) = q^{\binom{n}{2}} \sum_{i=0}^n
(-1)^i \cdot [n-i]!_q \cdot M_i(B,q).
\]
\end{corollary*}
This formula is used to recover the formula
in \cite{LLMPSZ} for the number $D_n(q)$ of $n\times n$ invertible matrices
with entries in $\Fq$ with zero diagonal
(a $q$-analogue of derangements); see \cite[Cor.~52]{Ravagnani}. We use
it to find a $q$-analogue of the {\em m\'enage problem} (Theorem~\ref{thm:qmenage}), 
settling a question considered by Rota and Haglund (private communication from Haglund).
Our $q$-analogue is very similar to Touchard's classical formula
\eqref{eq:menage} for $c_n$.

The starting point of any nice result in rook theory
is the case of Ferrers boards \cite{GJW, GarsiaRemmel, Haglund}.
In \cite{KleinLM} and \cite{LM}, we studied the matrix counts 
$M_i(B,q)$ for a richer class of boards, 
the \emph{(coinversion) diagrams} of permutations (see Section~\ref{sec:notation} for the definition). 
Our study included the following conjecture.

\begin{conjecture}[{\cite[Conj.~5.1]{KleinLM}}]
\label{rothe conjecture}
For all permutations $w\in \Sym_n$ and ranks $0\leq r \leq n$, the reduced
matrix count $M_r(\overline{I_w},q)$
of $n\times n$ matrices over $\Fq$ of rank $r$ with support in
the complement of the diagram $I_w$ of $w$
is a polynomial in $q$ with nonnegative integer coefficients. 
\end{conjecture}

We verified the conjecture computationally for $r\leq n\leq 7$ and for $r=n$ for $n \leq 9$ \cite{LMsite}. 
In \cite{LM}, we proved the conjecture for permutations $w$
 avoiding the patterns $4231$, $35142$, $42513$ and $351624$ in the case $r = n$. In the second part of this paper, we use the complement identity to prove
 the polynomiality part of Conjecture~\ref{rothe conjecture} (Corollary~\ref{cor:polynomiality-rothe}).

\begin{theorem*}
For all permutations $w\in \Sym_n$ and all ranks $0\leq r\leq n$,
$M_r(\overline{I_w},q)$ is a polynomial in $q$ with integer coefficients. 
\end{theorem*}

We also give a deletion-contraction relation (Corollary~\ref{cor:delconw})
that allows for the quick computation of $M_r(\overline{I_w},q)$.  Using this relation, we find counterexamples to the positivity part of Conjecture~\ref{rothe conjecture}.

\begin{example}
For $w =6\,8\,9\,10\,4\,5\,7\,1\,2\,3 \in \Sym_{10}$; we have
\[
M_{10}(\overline{I_w},q) = q^{77} + 9q^{76} + 44q^{75} + \cdots + 2q^{48} - 8q^{47} - q^{46} + q^{45}\not\in \mathbb{N}[q].
\]
\end{example}

It remains open to characterize the permutations $w$ such that
$M_r(\overline{I_w},q)$ is in $\mathbb{N}[q]$.

\subsection*{Outline} Section~\ref{sec:notation} establishes notation and introduces a $q$-analogue of the MacWilliams complement identity from the literature.  Section~\ref{sec:qhit} introduces a $q$-analogue of the hit numbers and proves a variety of properties analogous to the classical case.  Section~\ref{sec:NE property} studies the $q$-rook and $q$-hit numbers for boards with a certain structural property, including connections to the Garsia--Remmel $q$-rook theory and a $q$-analogue of the \emph{probl\`eme des m\'enages}.  Section~\ref{sec:deletion-contraction} builds on Section~\ref{sec:NE property} to give deletion-contraction style recurrences for $q$-rook and $q$-hit numbers.  Finally, Section~\ref{sec:final} includes a number of additional remarks and open questions.

\subsection*{Acknowledgements}
We are grateful to Jim Haglund and Igor Pak for helpful conversations.
We thank Dennis Stanton for his valuable insights into $q$-series and
\K polynomials.  Finally, we are indebted to Jeffrey Remmel, from
whose crucial suggestions this project initially grew.  

JBL was supported in part by NSF grant DMS-1401792. 
AHM was supported in part by an AMS--Simons Foundation travel grant.

\section{Background and notation}
\label{sec:notation}

Throughout this paper, $m$ and $n$ will be fixed positive integers with $m \leq n$.
Given an integer $k$, denote by $[k]$ the set $\{1, \ldots, k\}$ of the first $k$ positive integers.  
We use the word \emph{board} to refer to any subset of $[m] \times [n]$.  Given a board $B$, we denote by $\overline{B}$ its \emph{complement} $\overline{B} \defeq ([m] \times [n]) \smallsetminus B$ in the rectangle $[m] \times [n]$.  A \emph{rook placement} on a board $B$ is a subset of $B$ that contains no two elements in the same row (i.e., having the same first coordinate) or in the same column (having the same second coordinate).  When drawing boards and rook placements, we always use matrix coordinates, so that $\{(1, y) : y \in [n]\}$ is the top-most row and $\{(x, 1) : x \in [m]\}$ is the left-most column.  The elements of a board $B$ will be variously referred to as \emph{cells} or \emph{boxes}.

One particularly nice family of boards are the \emph{Ferrers boards}.
Each Ferrers board is associated to an \emph{integer partition}
$\lambda = (\lambda_1 \geq \lambda_2 \geq \cdots \geq \lambda_k > 0)$,
and consists of an aligned collection of $\lambda_i$ boxes in the
$i$th row for $i = 1, \ldots, k$.  We take an ecumenical approach and
use the name Ferrers board for boards in both English and French
notation, as well as their reflections.

We make use of many standard notations for $q$-counting functions, including the \emph{$q$-Pochhammer symbol} and \emph{$q$-factorial}
\[
(a; q)_k \defeq \prod_{i = 0}^{k - 1} (1 - aq^i) = (aq^{k - 1}; q^{-1})_k
\qquad\textrm{ and }\qquad
[k]!_q \defeq \frac{(q; q)_k}{(1 - q)^k} = \prod_{i = 1}^k \frac{q^i - 1}{q - 1},
\]
and the \emph{$q$-binomial coefficient}, defined by
\[
\qbin{k}{\ell}{q} \defeq \frac{[k]!_q}{[\ell]!_q \cdot [k - \ell]!_q} \qquad \textrm{ if } 0 \leq \ell \leq k
\]
and $\qbin{k}{\ell}{q} \defeq 0$ otherwise.  It is not obvious from this definition, but the $q$-binomial coefficients are polynomials in $q$ with positive integer coefficients. They also give the expansion of the $q$-Pochhammer as a sum, called the \emph{$q$-binomial theorem}.  In its most general form \cite[(II.3)]{GasperRahman}, this is the infinite product-sum identity
\begin{equation}
\label{eq:general q-binomial theorem}
\frac{(az; q)_\infty}{(z; q)_\infty} = \sum_{i = 0}^\infty \frac{(a; q)_i}{(q; q)_i} z^i,
\end{equation}
but on specializing $a \mapsto q^{-k}$ and $z \mapsto q^kz$ for $k \in \NN$ it becomes
\begin{equation}
\label{eq:q-binomial theorem}
(z; q)_k = \sum_{i = 0}^k (-1)^i \qbin{k}{i}{q} q^{\binom{i}{2}} z^i.
\end{equation}
The inverse relation
\begin{equation}
\label{eq:inverse q-binomial}
z^k = \sum_{i = 0}^k (-1)^i  \qbin{k}{i}{q} q^{\binom{i}{2}} (z; q^{-1})_i
\end{equation}
expressing the pure powers of $z$ in terms of $q$-Pochhammer symbols may be proved by expanding the $q$-Pochhammer in the right side of \eqref{eq:inverse q-binomial} using \eqref{eq:q-binomial theorem}, reversing the order of summation, and re-collecting terms in the inner sum using \eqref{eq:q-binomial theorem}.

Given a board $B$, let $\mat_i(B, q)$ be the number of $m\times n$
matrices of rank $i$ over $\F_q$ with support in $B$ (that is, with
all entries outside of $B$ equal to $0$), and let $M_i(B, q) \defeq
\mat_i(B, q)/(q-1)^i$. In \cite[Prop.~5.1]{LLMPSZ}, we showed that
$M_i(B,q)$ is an enumerative $q$-analogue of $r_i(B)$, in the following sense: for any prime power $q$,
\begin{equation} \label{eq:Mq-analogue}
M_i(B, q) \equiv r_i(B) \pmod{q-1}.
\end{equation}

The \emph{symmetric group} $\Sym_n$ consists of the permutations of the set $[n]$.  These may be represented in various ways: as words $w = w_1 \cdots w_n$ in one-line notation, or as permutation matrices, having entries $1$ at positions $(i, w_i)$ for $i \in [n]$ and other entries $0$, or as placements of $n$ rooks on $[n] \times [n]$.

For a permutation $w=w_1\cdots w_n$ in $\Sym_n$, define its \emph{diagram}\footnote{There are many possible variations on the diagram $I_w$: different choices of coordinates for $w$ give different correspondences between the set of permutations and the set of their diagrams, or amount to reflecting or rotating the diagrams; recording the pairs $(i, j)$ instead of $(i, w_j)$ produces diagrams with permuted columns; recording inversions instead of coinversions is equivalent to recoordinatizing; and so on.  None of these differences materially affect our results.  Appropriate variations are known in the literature as \emph{inversion diagrams} or \emph{Rothe diagrams} of permutations.}
$I_w$ by $I_w \defeq \{(i, w_j) \mid i<j, w_i<w_j\}$. 
The cells of $I_w$ are in bijection with the \emph{coinversions} of
$w$, and so $|I_w| = \binom{n}{2} - \ell(w)$ where $\ell(w)$ is the
length (or inversion number) of $w$.  The permutation boards contain
the Ferrers boards as a sub-class: any Ferrers board that fits inside
the upper-right-aligned triangle with legs of length $n - 1$ is the
diagram of a permutation in $\Sym_n$.  Conversely, a permutation $w$
has as its diagram an upper-right-aligned Ferrers board if and only if
$w$ avoids the \emph{permutation pattern} $312$, in the sense that $w$
has no three entries $w_i > w_k > w_j$ with $i < j < k$
\cite[Ex.~2.2.2]{Manivel}.

\subsection{A MacWilliams-style complement identity}

The classical MacWilliams identity expresses the 
weight of a \emph{code} (a subspace of a finite vector space) 
in terms of the weight of the dual code \cite{MacWilliams, survey}.  
In \cite{Delsarte}, Delsarte
introduced \emph{rank-metric codes}, in which the code is a linear
subspace of matrices over a finite field and the weight of an element
is the rank.\footnote{
  Note that the set of matrices supported on a given board is a linear
  subspace, and thus a code in this sense.}
In this context, he proved what may be viewed as a $q$-analogue of the MacWilliams identity \cite[Thms.~3.3, A2]{Delsarte}.  This identity involves a $q$-analogue of the \emph{\K polynomials}, so we begin by recalling some important facts about them from the literature.

For $i, r \leq m$, define the \emph{$q$-\K polynomial}\footnote{
  Regrettably, there are several families of polynomials that go 
  by this name; see, e.g., \cite[Ex.~7.8, 7.11]{GasperRahman} 
  and \cite[\S4]{Stanton}, where the polynomials related to ours 
  are the \emph{affine $q$-\K polynomials}.}
\[
K_r(i) \defeq \sum_{s} (-1)^{r-s} q^{ns+\binom{r-s}{2}} \qbin{m-s}{r-s}{q}\qbin{m-i}{s}{q},
\]
where the sum is over all indices $s$ such that $0 \leq r - s$ and $0 \leq s \leq m - i$.
These polynomials form a family of orthogonal polynomials, and consequently have many nice properties.  We mention several of these here, following to various degrees Delsarte, Ravagnani, and Stanton \cite{Delsarte, Ravagnani, Stanton}.
Let
\begin{equation} \label{eq:defvk}
v_k \defeq \prod_{i=0}^{k-1} \frac{(q^m - q^i)(q^n-q^i)}{q^k-q^i}
\end{equation}
denote the number of $m\times n$ matrices of rank $k$ over $\F_q$
(see, e.g., \cite[\S 1.7]{Morrison}). The $q$-\K polynomials satisfy the orthogonality relation
\[
\sum_{i=0}^n  v_i \cdot K_k(i) \cdot K_{\ell}(i) = q^{mn} \cdot v_k \cdot \delta_{k,\ell},
\]
where $\delta_{k, \ell}$ represents the usual Kronecker delta function.
They can be written in terms of basic hypergeometric functions in various ways; notably,
\[
K_r(i) =  v_r \cdot {}_3\phi_2( q^{-r}, q^{-i}, 0; \quad q^{-m}, q^{-n}; \quad q).
\]
(It is the ${}_3\phi_2$ evaluation on the right side that is called the \emph{affine $q$-\K polynomial} by Stanton \cite[(4.13)]{Stanton}; Delsarte \cite{Delsarte} gives a similar expression, but it contains an error.)  This formula exhibits the symmetry
\begin{equation}
\label{K symmetry}
\frac{K_r(i)}{v_r} = \frac{K_i(r)}{v_i},
\end{equation}
which is not obvious from the definition.
As orthogonal polynomials, the $q$-\K polynomials satisfy a three-term recurrence relation \cite[(4.14)]{Stanton}:
\begin{multline*}
q^{m + n} (q^{-k} - 1) K_i(k)  = 
q^i(q^{i + 1} - 1) K_{i + 1}(k) 
+ {}\\ 
(q^{m} - q^{i - 1})(q^{n} - q^{i - 1}) K_{i - 1}(k) 
-  
((q^m - q^i)(q^n - q^i) + q^{i - 1}(q^i - 1))K_i(k).
\end{multline*}

The relevance of the $q$-\K polynomials to the present work is their
appearance in the following complement identity, expressing the number
of matrices of a given rank supported on a board in terms of the same
counts for the complementary board.

\begin{theorem}[{Complement identity \cite{Delsarte}}] \label{thm:MW}
For any board $B \subseteq [m] \times [n]$ with $m \leq n$ and any rank $r \leq m$, we have
\begin{equation} \label{eq:MW}
\mat_r(\overline{B}, q) = \frac{1}{q^{|B|}} \sum_{i=0}^m K_r(i) \cdot \mat_i(B,q).
\end{equation}
\end{theorem}

In the case $r=m$ of full-rank matrices, this formula simplifies.

\begin{corollary} \label{cor:fullrankrect}
For any subset $B$ of $[m]\times [n]$ with $m\leq n$, we have 
\begin{equation} 
\label{eq:fullrankrect}
\mat_m(\overline{B}, q) = (-1)^m q^{\binom{m}{2}-|B|} \sum_{i=0}^m
\mat_i(B,q)\cdot (q^{n-m+1}; q)_{m-i}.
\end{equation}
\end{corollary} 

\begin{proof}
The case $r=m$ in \eqref{eq:MW} gives
\begin{align*}
\mat_m(\overline{B}, q) &= \frac{1}{q^{|B|}} \sum_{i=0}^m
\mat_i(B,q) \cdot \sum_{s} (-1)^{m-s}q^{\binom{m}{2}+\binom{s}{2}+s(n-m+1)}
\qbin{m-i}{s}{q}\\
&= (-1)^m q^{\binom{m}{2}-|B|} \sum_{i=0}^m
\mat_i(B,q) \left(\sum_{s=0}^{m-i} (-q^{n-m+1})^s
  q^{\binom{s}{2}} \qbin{m-i}{s}{q}\right).
\end{align*}
By the $q$-binomial theorem \eqref{eq:q-binomial theorem}, the inner sum simplifies to 
\[
\sum_{s=0}^{m-i} (-q^{n-m+1})^s
  q^{\binom{s}{2}} \qbin{m-i}{s}{q} = (q^{n-m+1};q)_{m-i},
\]
as desired.
\end{proof}

\begin{example}[$q$-analogue of derangements \cite{LLMPSZ, Ravagnani}] \label{ex:qderangements}
In the case $m = n$ and $B=\{(1,1),\ldots,(n,n)\}$, we have that
$\mat_i(B,q)=\binom{n}{i}(q-1)^i$.  Thus, by \eqref{eq:fullrankrect},
the number of $n\times n$ invertible matrices with zero diagonal is 
\begin{align*}
\mat_n(\overline{B}, q) &= (-1)^n q^{\binom{n}{2}-n} \sum_{i=0}^n
  \binom{n}{i} (q-1)^i (q;q)_{n-i}\\
&= q^{\binom{n}{2}-n} (q-1)^n \sum_{i=0}^n (-1)^i \binom{n}{i} [n-i]!_q.
\end{align*}
\end{example}

In Section~\ref{sec:menage-examples}, we give a formula for
$\mat_n(\overline{B'}, q)$ for the board $B'$ consisting of the main
diagonal $\{(1, 1), \ldots, (n, n)\}$ together with the next diagonal
$\{(1, 2), \ldots, (n - 1, n)\}$ and the entry $\{(n,1)\}$.  
The classical rook theory of the board 
$B'$ is the famous {\em probl\`eme des m\'enages}.

\section{$q$-hit numbers} \label{sec:qhit}

Consider a board $B$ contained in the rectangle $[m]\times [n]$.  In
this case, the
version of the 
classical relation \eqref{eq:classic_rookhit_rel} between hit numbers and rook numbers is
\begin{equation} \label{eq:classichitrookrect}
\sum_{i=0}^m h_i(B) \cdot t^i = \sum_{i=0}^m r_i(B) \cdot \frac{(n-i)!}{(n-m)!} \cdot (t-1)^i.
\end{equation}
This relation follows by the same argument as in the square case
(see, e.g., \cite[\S2.3]{EC1}).

We now define our $q$-hit numbers for general boards $B$. 
This definition is based on a suggestion of Remmel
(private communication) that is related to the construction of the
Garsia--Remmel $q$-hit numbers \cite[\S 1]{GarsiaRemmel}. 

\begin{definition}
Given a board $B\subseteq[m]\times [n]$ and a nonnegative integer $i$, define the \emph{$q$-hit number} $\nHit_i(B,q)$ by the equation
\begin{equation}
\label{definition of hits inside rectangle}
\sum_{i=0}^m \nHit_i(B,q) \cdot t^i 
\defeq 
q^{\binom{m}{2}}
\sum_{i=0}^m M_i(B,q) \cdot \frac{[n-i]!_q}{[n - m]!_q} \cdot (-1)^i
                           \cdot (t;q^{-1})_i.
\end{equation}
For each fixed $q$, both sides of this equation are polynomials in
$t$. We call this the \emph{$q$-hit polynomial} of
the board $B$ and denote it by $P(B, q, t)$.
\end{definition}

The remainder of this section is devoted to showing that these $q$-hit numbers
satisfy a variety of properties that one would expect from an object bearing
the name; this culminates in a natural probabilistic interpretation of the
$q$-hit numbers in Section~\ref{sec:RemmelHitIdea}.

\begin{example}
\label{ex:negpolyhit}
We give three examples of boards in the case $m = n = 2$; these are illustrated in Figure~\ref{fig:2 by 2 boards}.
\begin{figure}
\begin{center}
\includegraphics[scale=1.5]{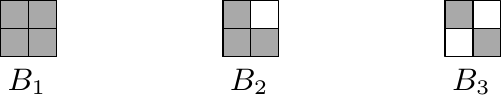}
\end{center}
\vspace{-.2in}
\caption{The three $2 \times 2$ boards mentioned in Example~\ref{ex:negpolyhit}.}
\label{fig:2 by 2 boards}
\end{figure}
\begin{compactenum}[(i)]
\item When $B_1 = [2] \times [2]$ is the entire square board, we have $M_0(B_1, q) = 1$, $M_1(B_1, q) = (q + 1)^2$, and $M_2(B_1, q) = q(q + 1)$.  Thus
\[
{P}(B_1, q, t) = q \left( (q + 1) + (q + 1)^2 \cdot (t - 1) + q(q + 1) \cdot (t - 1)(tq^{-1} - 1)\right) = (q^2 + q)t^2,
\]
and so $\nHit_0(B_1, q) = \nHit_1(B_1, q) = 0$ and $\nHit_2(B_1, q) = q^2 + q$.

\item When $B_2 = \overline{\{(1, 2)\}}$ is the $2 \times 2$ board with a single square removed, we have $M_0(B_2, q) = 1$, $M_1(B_2, q) = 2q + 1$, and $M_2(B_2, q) = q$.  Thus
\[
{P}(B_2, q, t) = q \left( (q + 1) + (2q + 1) \cdot (t - 1) + q \cdot (t - 1)(tq^{-1} - 1)\right) = q^2t + qt^2,
\]
and so $\nHit_0(B_2, q) = 0$, $\nHit_1(B_2, q) = q^2$, and $\nHit_2(B_2, q) = q$.

\item When $B_3 = \{(1, 1), (2, 2)\}$ comprises the two diagonal squares, we have $M_0(B_3, q) = 1$, $M_1(B_3, q) = 2$, and $M_2(B_3, q) = 1$.  Thus
\[
{P}(B_3, q, t) = q \left( (q + 1)  + 2 \cdot (t - 1) +  (t - 1)(tq^{-1} - 1)\right) = q^2 + (q - 1)t + t^2,
\]
and so $\nHit_0(B_3, q) = q^2$, $\nHit_1(B_2, q) = q - 1$, and $\nHit_2(B_2, q) = 1$.
\end{compactenum}
\end{example}

\subsection{Basic properties}
\label{sec:basic properties}

In this section, we establish several basic properties of $q$-hit numbers.
We begin by showing that $q$-hit numbers are enumerative
$q$-analogues of the classical hit numbers $h_i(B)$, justifying 
our choice of name.
 
\begin{proposition} \label{prop:qhitqanalogue}
Fix a prime power $q$.  For all $B\subseteq[m]\times [n]$ with $m\leq n$ and $i=0, 1,\ldots,m$, we have that 
\[
\nHit_i(B,q) \equiv  h_i(B) \pmod{q-1}. 
\]
\end{proposition}

\begin{proof}
We start with \eqref{definition of hits inside rectangle} and take the residue
modulo $q-1$. By \eqref{eq:Mq-analogue}, we have for each $i$ 
that $M_i(B,q) \equiv r_i(B) \pmod{q-1}$, and thus  
\[
\sum_{i=0}^n \nHit_i(B,q) \cdot t^{i}  \equiv  
\sum_{i=0}^m r_i(B) \cdot \frac{(n - i)!}{(n-m)!} \cdot (t-1)^i
\pmod{q-1}.
\]
The right side of this equivalence is the generating function \eqref{eq:classichitrookrect} for the classical hit numbers, so
\[
\sum_{i=0}^n \nHit_i(B,q) \cdot t^{i}  \equiv  \sum_{i=0}^m h_i(B) \cdot 
t^i \pmod{q-1}.
\]
By equating coefficients of $t^i$ for $i=0,\ldots,m$, we obtain the
desired result.
\end{proof}

\begin{remark}
For the $q$-rook numbers, we have that $M_i(B, q) = 0$ if and only if the usual rook number $r_i(B)$ is also equal to $0$.  However, this does not hold for $q$-hit numbers.  For example, in Example~\ref{ex:negpolyhit}(iii)
with $m = n = 2$ and $B=\{(1,1),(2,2)\}$, we have $h_1(B) = 0$ (every permutation has either $0$ or $2$ rooks on $B$) but $\nHit_1(B, q) = q-1$.
\end{remark}

In the classical setting, maximal rook placements on $B$ are exactly
the same as rook placements in which all rooks land on $B$, and so
$r_m(B) = h_m(B)$. Moreover, from \eqref{eq:classic_rookhit_rel} one
can write each rook number in terms of hit numbers and vice versa. The
next proposition shows that the analogous results hold for $q$-hit numbers and
matrix counts.

\begin{proposition} \label{prop:hitasMs}
For all $B\subseteq[m]\times [n]$ and for all $k=0,1,\ldots,m$ we have that
\begin{align*}
\nHit_k(B,q)  &= q^{\binom{k + 1}{2} + \binom{m}{2}} \sum_{i=k}^{m} M_i(B,q) \cdot
\frac{[n-i]!_q}{[n-m]!_q} \qbin{i}{k}{q} (-1)^{i + k}
                q^{-ik}\\
\intertext{and}
M_k(B, q) &= q^{\binom{k}{2} -\binom{m}{2}} \frac{[n-m]!_q}{[n-k]!_q}
           \sum_{i=k}^m \nHit_i(B,q) \qbin{i}{k}{q}.
\end{align*}
In particular,
\[
\nHit_m(B,q) = M_m(B, q).
\]
\end{proposition}

\begin{proof}
The relations follow by extracting the coefficients of $t^k$ and
$(t;q^{-1})_k$ respectively from both sides
of \eqref{definition of hits inside rectangle} using the 
$q$-binomial theorem \eqref{eq:q-binomial theorem} and its inverse transformation 
\eqref{eq:inverse q-binomial}, and rearranging powers of $q$.
\end{proof}

Another straightforward result in the classical case is that $\sum_i h_i(B) = n(n - 1) \cdots (n - m + 1)$, since both sides count the total number of maximal non-attacking rook placements on $[m] \times [n]$.  The next result is the analogue in our setting.

\begin{corollary} \label{cor:hit partition}
 For all $B\subseteq[m]\times [n]$, we have
\[
(q-1)^m\sum_{i=0}^m \nHit_i(B,q) = v_m =(q^n-1)(q^n-q)\cdots(q^n-q^{m-1}).
\]
\end{corollary}
\begin{proof}
Set $k=0$ in the second equation in Proposition~\ref{prop:hitasMs} to obtain
\begin{align*}
\sum_{i=0}^m \nHit_i(B,q)  &= 
q^{\binom{m}{2}} M_0(B,q) \cdot \frac{[n]!_q}{[n-m]!_q} \\
&= q^{\binom{m}{2}} \frac{[n]!_q}{[n-m]!_q} = \frac{(q^n-1)(q^n-q)\cdots (q^n-q^{m-1})}{(q-1)^m},
\end{align*}
as claimed.
\end{proof}

This proposition is particularly suggestive when $m = n$, and the right side becomes $|\GL_n(\F_q)|$ -- see Question~\ref{open question:interpretation} below.  
We end this section with a final property that $q$-hit numbers share with classical hit numbers.

\begin{proposition} \label{prop:invariance}
For all $B\subseteq[m]\times [n]$ and for all $i=0,1,\ldots,m$,
the $q$-hit number $\nHit_i(B,q)$ is invariant under permuting rows and columns of $B$.
\end{proposition}
\begin{proof}
The numbers $M_i(B,q)$ are invariant under permuting rows and columns
of $B$, so the result follows immediately from \eqref{definition of hits inside rectangle}.
\end{proof}

\subsection{Reciprocity}
\label{sec:reciprocity}

In this section, we use the complement identity (Theorem~\ref{thm:MW}) to prove a reciprocity theorem for $q$-hit polynomials analogous to the classical \eqref{eq:classic-hit-recip}.  We begin with a technical lemma.

\begin{lemma}
\label{polynomial lemma}
For any positive integers $m \leq n$ and nonnegative integer $i \leq m$, we have
\begin{equation}
\label{eq:K identity}
\sum_{r=0}^m K_{r}(i)\cdot (q^{n - m + 1}; q)_{m - r} \cdot (t; q^{-1})_r
=
t^m \cdot (q^{n -m + 1}; q)_{m-i} \cdot (t^{-1}q^{n}; q^{-1})_i.
\end{equation}
\end{lemma}
\begin{proof}
Denote by $L$ the left side of the identity to be proved.
Using the definition of the $q$-\K polynomials and reversing the order of summation gives
\begin{align*}
L & = \sum_{r=0}^m \sum_{s = 0}^{\min(r, m - i)} (-1)^{r-s} q^{ns+\binom{r-s}{2}} \qbin{m-s}{r-s}{q}\qbin{m-i}{s}{q}(q^{n - m + 1}; q)_{m - r} \cdot (t; q^{-1})_r \\
& =  \sum_{s = 0}^{m - i} q^{ns}\qbin{m-i}{s}{q} (q^{n - s}; q^{-1})_{m - s}(t; q^{-1})_s
     \sum_{r = s}^m q^{s - r} 
         \frac{(q^{m - s}; q^{-1})_{r - s} (tq^{-s}; q^{-1})_{r - s}}{(q^{-1}; q^{-1})_{r - s} (q^{n - s}; q^{-1})_{r - s}}.
\end{align*}
The $q$-Chu--Vandermonde identity \cite[(II.6)]{GasperRahman} asserts
\[
\sum_{k} \frac{(a; x)_k (x^{-N}; x)_k}{(c; x)_k (x; x)_k} x^k = \frac{(c/a; x)_N}{(c; x)_N} a^N.
\]
Setting $(a, c, x, N) \mapsto (tq^{-s}, q^{n - s}, q^{-1}, m - s)$, this implies
\begin{align*}
L & = \sum_{s = 0}^{m - i} q^{ns}\qbin{m-i}{s}{q} (q^{n - s}; q^{-1})_{m - s}(t; q^{-1})_s \cdot
     \frac{(t^{-1}q^n; q^{-1})_{m - s}}{(q^{n - s}; q^{-1})_{m - s}} t^{m - s} q^{-s(m - s)} \\
  & = t^m \cdot (t^{-1}q^n; q^{-1})_m \sum_{s = 0}^{m - i} \frac{(q^{m - i}; q^{-1})_{s} (t; q^{-1})_s}{  (q^{-1}; q^{-1})_{s} (t q^{-n +m - 1}; q^{-1})_s} q^{-s}.
\end{align*}
Applying $q$-Chu--Vandermonde again with $(a, c, x, N) \mapsto (t, tq^{-n + m - 1}, q^{-1}, m - i)$ gives
\[
L =  t^m \cdot (t^{-1}q^n; q^{-1})_m \frac{(q^{-n + m - 1}; q^{-1})_{m - i}}{(tq^{-n + m - 1}; q^{-1})_{m - i}} t^{m - i} = t^m \cdot (t^{-1}q^n; q^{-1})_i  \cdot (q^{n - m + 1}; q)_{m - i},
\]
as claimed.
\end{proof}

\begin{theorem}[Reciprocity] \label{thm:ourhitreciprocity}
For any board $B\subseteq[m] \times [n]$, the $q$-hit numbers of $B$ satisfy
\begin{equation} \label{eq:hitreciprocity1}
\nHit_{m-i}(\overline{B},q) =  q^{in-|B|} \cdot \nHit_i(B,q),
\end{equation}
or equivalently
\begin{equation}\label{eq:hitreciprocity2}
{P}(\overline{B}, q, t) = q^{-|B|}t^m \cdot {P}(B, q, q^{n}/t).
\end{equation}
\end{theorem}
\begin{proof}
First we prove \eqref{eq:hitreciprocity2}.  Applying Theorem~\ref{thm:MW} to the definition \eqref{definition of hits inside rectangle} gives
\begin{align*}
q^{|B|}(q - 1)^m \sum_{i=0}^m \nHit_i(\overline{B},q)t^{i} 
&= (-1)^m q^{\binom{m}{2}} 
   \sum_{r=0}^m q^{|B|}\mat_r(\overline{B}, q) \cdot (q^{n - m + 1}; q)_{m - r} \cdot (t; q^{-1})_r \\ 
&= (-1)^m q^{\binom{m}{2} } \sum_{i=0}^m \mat_i(B,q) \sum_{r=0}^m K_{r}(i)\cdot (q^{n - m + 1}; q)_{m - r} \cdot (t; q^{-1})_r.
\end{align*}
By Lemma~\ref{polynomial lemma}, we may rewrite the last expression to give
\[
q^{|B|} (q - 1)^m \cdot {P}(\overline{B}, q, t) = 
(-1)^m q^{\binom{m}{2}} t^m \sum_{i=0}^m \mat_i(B,q) \cdot (q^{n -m + 1}; q)_{m-i} \cdot (t^{-1}q^{n}; q^{-1})_i.
\]
Dividing by $q^{|B|}(q - 1)^m$ and comparing with the definition \eqref{definition of hits inside rectangle}, one immediately sees the result.  Finally, \eqref{eq:hitreciprocity1} follows by extracting the coefficient of $t^{m - i}$ in \eqref{eq:hitreciprocity2}.
\end{proof}

The classical zeroth hit number $h_0$ satisfies two simple formulas: 
first, for any board $B$, we have $h_0(B) = r_m(\overline{B})$,
as both sides count maximum-rank rook placements with all rooks
outside $B$.  Second, extracting the constant term from both sides of
\eqref{eq:classichitrookrect} gives the inclusion-exclusion formula
\[
h_0(B) = \sum_{i = 0}^m (-1)^i \cdot \frac{(n - i)!}{(n - m)!} \cdot r_i(B).
\]
Our next result is a $q$-analogue of these formulas.

\begin{corollary}
\label{cor:H_0 as inclusion-exclusion}
For any board $B\subseteq[m]\times [n]$ with $m\leq n$, we have
\[
\nHit_0(B,q) = q^{|B|}M_m(\overline{B}, q)
=  q^{\binom{m}{2}} \sum_{i=0}^m (-1)^i \cdot \frac{[n - i]!_q}{[n - m]!_q} \cdot M_i(B,q).
\]
\end{corollary}
\begin{proof}
By the $q$-hit reciprocity \eqref{eq:hitreciprocity1} 
we have that $\nHit_0(B,q) = q^{|B|} \cdot \nHit_m(\overline{B},q)$.  
Also, by Proposition~\ref{prop:hitasMs} we have that
$\nHit_m(\overline{B}, q)= M_m(\overline{B}, q)$. 
Combining these two gives the first formula.
For the second formula, we use 
\eqref{eq:fullrankrect} to evaluate $M_m(\overline{B}, q)$.
\end{proof}

\subsection{Probabilistic interpretation} \label{sec:RemmelHitIdea}

In this section, we give a probabilistic interpretation to the $q$-hit polynomial ${P}(B, q, t)$.  Consider a board $B$ contained in the rectangle $[m] \times [n]$, and let $B_k$ be the board that we get by adding $k$ rows of length $n$ below $B$, as in Figure~\ref{fig:extended board}.  Let $F_k(B, q) \defeq \mat_m(B_k, q)/q^{nk + |B|}$ be the probability that a random matrix with support on $B_k$ has rank $m$.  There is a natural generating function
\[
{F}_{\infty}(B, q, t) \defeq \sum_{k = 0}^\infty t^k F_k(B, q)
\]
for these numbers, which we may think of as counting \emph{infinite} matrices by the number of their rows at which they first achieve rank $m$.
The main result of this section shows a close relation between ${F}_\infty$ and ${P}(B, q, t)$, and so provides a probabilistic interpretation for the $q$-hit polynomial.

\begin{figure}
\includegraphics{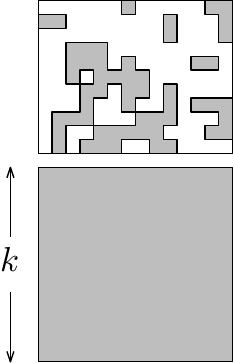}
\caption{Extending a board $B$ by $k$ rows.}
\label{fig:extended board}
\end{figure}

\begin{theorem} \label{thm:probinterpret}
For any board $B$, we have that the generating function of the
probabilities $F_k(B, q)$ is given by
\[
{F}_{\infty}(B, q, t) = \frac{ q^{-|B|- mn} t^m (q - 1)^m}{(tq^{-n};q)_{m + 1}}  {P}(B, q, q^{n}t^{-1}).
\]
\end{theorem}
\begin{proof}
We begin by giving an alternate expression for $F_k(B, q)$.  For each matrix $X$ of rank $r$ with support in $B$, the number of ways to extend $X$ to a matrix of rank $m$ with support in $B_k$ is exactly
\[
q^{rk} \cdot \#\{ k \times (n - r) \textrm{ matrices } Y \textrm{ with } \rk(Y) = m - r \}.
\]
(This fact may be explained in more or less sophisticated language; at its simplest, it follows easily after multiplication on the right by an invertible matrix to put the $m \times n$ block in reduced column echelon form.)  The second factor is an instance of the expression we denote $v_{m - r}$.  It is equal to $0$ if $k - m + r < 0$, and by some easy manipulations of $q$-Pochhammer symbols we may write it as $(-1)^{m - r} q^{\binom{m - r}{2}} \frac{(q^{n - r}; q^{-1})_{m - r} \cdot (q^{m - r + 1}; q)_{k - m + r}}{(q; q)_{k - m + r}}$ otherwise.
On the other hand, the number of choices of $X$ is simply $\mat_r(B, q)$, so summing over all choices of $r$ we have
\[
F_k(B, q) = q^{-nk - |B|} \sum_{r = m - k}^m \mat_r(B, q) \cdot (-1)^{m - r} q^{rk + \binom{m - r}{2}} \frac{(q^{n - r}; q^{-1})_{m - r} \cdot (q^{m - r + 1}; q)_{k - m + r}}{(q; q)_{k - m + r}}.
\]
Plugging this in to the definition of ${F}_\infty$ and rearranging yields
\begin{align*}
{F}_\infty & = \sum_{k = 0}^\infty t^k q^{-nk - |B|} \left(\sum_{r = m - k}^m (-1)^{m - r} \mat_r(B, q)  q^{rk + \binom{m - r}{2}} \frac{(q^{n - r}; q^{-1})_{m - r} \cdot (q^{m - r + 1}; q)_{k - m + r}}{(q; q)_{k - m + r}}\right) \\
& = q^{ - |B|} \sum_{r = 0}^m (-1)^{m - r} \mat_r(B, q) q^{\binom{m - r}{2}} (q^{n - r}; q^{-1})_{m - r}\left( \sum_{k = m - r}^\infty t^k q^{(r-n)k} \frac{ (q^{m - r + 1}; q)_{k - m + r}}{(q; q)_{k - m + r}} \right).
\end{align*}
Up to a power of $tq^{r - n}$, the inner sum is equal to the summation side of \eqref{eq:general q-binomial theorem} upon substituting $a \mapsto q^{m - r + 1}$, $z \mapsto tq^{r - n}$ and $i \mapsto k - m + r$.  Thus,
\begin{align*}
{F}_\infty  & = 
  q^{ - |B|} \sum_{r = 0}^m (-1)^{m - r} 
  \mat_r(B, q) q^{\binom{m - r}{2}} (q^{n - r}; q^{-1})_{m - r}
  \cdot(tq^{r - n})^{m - r} 
  \cdot \frac{(tq^{m - n + 1}; q)_\infty}{(tq^{r - n}; q)_\infty} \\
 & = 
  q^{ - |B|} \sum_{r = 0}^m (-1)^{m - r} 
  \mat_r(B, q)  (q^{n - r}; q^{-1})_{m - r}
  \cdot \frac{t^{m - r}q^{\binom{m - r}{2} + (r - n)(m - r)}}{(tq^{r - n}; q)_{m - r + 1}}.
\end{align*}
Putting this over a common denominator of $(tq^{-n};q)_{m + 1}$ and rearranging powers of $q$ and $t$ gives
\begin{align*}
{F}_\infty &= \frac{ q^{\binom{m}{2} -|B| - mn}}{(tq^{-n};q)_{m + 1}}  \sum_{r = 0}^m  (-1)^{m - r} \mat_r(B, q) \cdot t^{m - r} (q^{n - m + 1}; q)_{m - r} \cdot q^{ - \binom{r}{2} + nr} \frac{(tq^{-n};q)_{m + 1}}{ (tq^{r - n};q)_{m - r + 1}} \\
& = \frac{ q^{\binom{m}{2} -|B| - mn} t^m (q - 1)^m}{(tq^{-n};q)_{m + 1}}  \sum_{r = 0}^m  M_r(B, q) \cdot \frac{[n - r]!_q}{[n - m]!_q} \cdot (-1)^r \cdot (t^{-1}q^{n}; q^{-1})_r \\
& = \frac{ q^{-|B|-mn} t^m (q - 1)^m}{(tq^{-n};q)_{m + 1}} {P}(B, q, q^nt^{-1}),
\end{align*}
as claimed. 
\end{proof}

\begin{remark}
One could do the same computation for any rank $R$ between $m$ and $n$, inclusive (above we calculated with $R = m$).  The calculations are not substantially different; only the relatively tame factor in front changes.
\end{remark}

\begin{remark}
Garsia and Remmel obtained an analogue \cite[(I.12)]{GarsiaRemmel} 
of $F_{\infty}(B,q,t)$ when $B$ is a Ferrers board by considering
rook placements on an extended board.  In fact, their relation can be obtained from
Theorem~\ref{thm:probinterpret} using the results from Section~\ref{sec:NE property} below.
\end{remark}

\section{Boards with the NE property and their complements}
\label{sec:NE property}

\subsection{Garsia--Remmel $q$-rook numbers}

Given a placement $c$ of $r$ non-attacking rooks on a board $B$, Garsia and 
Remmel \cite{GarsiaRemmel} defined a {\em NE inversion} of $c$
to be a cell in $B$ that is not directly north or 
directly east of a rook in $c$.  Denote by $\inv^{\NE}_B(c)$ 
the number of NE inversions of the rook placement $c$.  This statistic
gives rise to a $q$-analogue 
\begin{equation} \label{eq:GRqrooknumber}
R^{\NE}_{r}(B,q) \defeq \sum_c q^{\inv_B^{\NE}(c)}
\end{equation}
of the rook number,
where the sum is over placements $c$ of $r$ non-attacking rooks on $B$.

A board $B\subseteq[m] \times [n]$ is said to have the \emph{NE property} 
if for all $i,i' \in [m]$ and $j,j' \in [n]$ such that $i<i'$ and $j<j'$,
we have that if $(i,j), (i',j)$, and $(i',j')$ are in $B$ then $(i,j')$ 
is also in $B$:
\begin{center}
\includegraphics{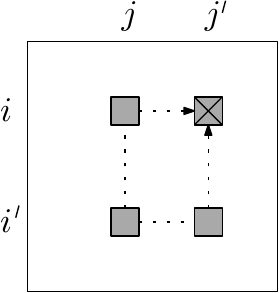}
\end{center}
These boards are 
convenient to work with, for the following reason: they are precisely the boards $B$ such that whenever the product $U_1 \cdot w \cdot U_2$, involving the rook placement $w$ on $[m] \times [n]$ and two upper-triangular matrices $U_1, U_2$ of respective sizes $m \times m$ and $n \times n$, has support on $B$, then $w$ is supported on $B$.  This means that the Bruhat decomposition, or equivalently Gaussian elimination, plays very nicely with matrices supported on $B$.  This observation may be exploited to give the following result, connecting our $q$-rook numbers with those of Garsia and Remmel.

\begin{theorem}[{\cite[Thm.~1]{Haglund}, \cite[Thm.~4.2]{KleinLM}}] 
\label{thm:polyqNE}
Fix any board $B \subseteq[m] \times [n]$ with the NE
property and any positive integer $r$.  
The number of $m\times n$ matrices over $\F_q$ of rank
$r$ whose support is in $B$ is 
\[
\mat_r(B, q) = (q-1)^r q^{|B|-r} \cdot R_r^{\NE}(B, q^{-1}).
\]
\end{theorem}

In particular, we have in this case that $M_r(B, q) = q^{|B|-r} \cdot R_r^{\NE}(B, q^{-1})$ is a polynomial in $q$ with nonnegative integer coefficients.

\begin{corollary} \label{cor:polyNEshapes}
If $B\subseteq[m]\times [n]$ has the NE property then
$M_r(\overline{B}, q) \in \mathbb{Z}[q]$.
\end{corollary}
\begin{proof}
The coefficients in the complement identity \eqref{eq:MW} are polynomials 
in $q$ with integer coefficients, so $M_r(\overline{B}, q)$ 
is also a polynomial in $q$ with integer coefficients.
\end{proof}

As a special case, we settle a question from \cite[Ques.~5.6]{LLMPSZ}.\footnote{
  When comparing the statement there, note a notational conflict: 
  in \cite{LLMPSZ}, $\mat_q(n,B,r)$ counts matrices with 
  support in the \emph{complement} $\overline{B}$.}
Given two partitions $\lambda$ and $\mu$, the \emph{skew Ferrers board} $S_{\lambda/\mu}$ consists of those elements of the board of shape $\lambda$ that do not belong to the board of shape $\mu$, when the two are aligned together.

\begin{corollary}
Let $S_{\lambda/\mu} \subseteq[m]\times [n]$ be a skew shape and $0\leq r \leq m$.  Then $M_r(\overline{S_{\lambda/\mu}}, q) \in \mathbb{Z}[q]$.
\end{corollary}
\begin{proof}
The board $S_{\lambda/\mu}$ (in a suitable orientation) has the
NE property and so the result follows from Corollary~\ref{cor:polyNEshapes}.
\end{proof}

We also get a polynomiality result for $q$-hit numbers.

\begin{corollary}
If $B\subseteq[m]\times [n]$ has the NE property then
$\nHit_i(B,q)$ and $\nHit_i(\overline{B},q)$ are in $\mathbb{Z}[q]$.
\end{corollary}

\begin{proof}
This follows by combining Corollary~\ref{cor:polyNEshapes},
Proposition~\ref{prop:hitasMs}, and Theorem~\ref{thm:ourhitreciprocity}.
\end{proof}

Recall that every permutation $w$ in $\Sym_n$ has an associated diagram 
\[
I_w=\{(i,w_j) \mid i<j, w_i < w_j\} \subseteq [n]\times [n].
\] 
It is easy to see that permutation diagrams have the NE property. Thus by
Theorem~\ref{thm:polyqNE} we have the following result.

\begin{corollary} \label{cor:polycoinv}
For any $n$ and $r$ and any permutation $w$ of size $n$, the
number of $n\times n$ matrices over $\F_q$ of rank $r$ whose support
is in $I_w$ is a polynomial:
\[
\mat_r(I_w, q)/(q-1)^r \in \mathbb{N}[q].
\]
\end{corollary}

Part of Conjecture~\ref{rothe conjecture} (originally posed in \cite{KleinLM}) states that $\mat_r(\overline{I_w}, q)$ is a
polynomial in $q$.  Combining
Corollary~\ref{cor:polycoinv} and the complementation formula
\eqref{eq:MW}, we can settle this part of the conjecture.

\begin{corollary} \label{cor:polynomiality-rothe}
For any $n$ and $r$ and any permutation $w$ of size $n$, the
number of $n\times n$ matrices over $\F_q$ of rank $r$ whose support
is in $\overline{I_w}$ is a polynomial:
\[
\mat_r(\overline{I_w}, q)/(q - 1)^r \in \mathbb{Z}[q].
\]
\end{corollary}

The positivity of $\mat_r(\overline{I_w}, q)$ is discussed below in Section~\ref{sec:failures of positivity}.

\subsection{Comparison with Garsia--Remmel $q$-hit numbers}

When $B=S_{\lambda}$ is a Ferrers board contained in the square $[n] \times [n]$, Garsia and Remmel defined a $q$-hit number $H^{\GR}_i(S_{\lambda},q)$.  These numbers are certain polynomials in $q$, defined\footnote{
  We use the definition by Haglund \cite[(3)]{Haglund} of
  these $q$-hit numbers as opposed to the original definition
  \cite[(2.1)]{GarsiaRemmel}. 
  The two definitions are equivalent up
  to dividing by $t^n$ and replacing $t$ by $1/t$.}
by the relation
\begin{equation} 
\sum_{i=0}^n H^{\GR}_i(S_{\lambda},q) t^i = \sum_{i=0}^n
                                           R_{i}^{\NE}(S_{\lambda},q)
                                           [n-i]!_q \prod_{k=n-i+1}^n (t-q^{k}). \label{HaglundRemmelversion}
\end{equation}
Garsia--Remmel showed that $H^{\GR}_i(S_{\lambda},q)$ is a polynomial
with nonnegative coefficients.

\begin{theorem}[{\cite[Thm.~2.1]{GarsiaRemmel}}] \label{thm:qposHit}
For any Ferrers board $S_{\lambda} \subset [n]\times [n]$ and $i=0,1,\ldots,n$
we have that $H^{\GR}_i(S_{\lambda},q)$ is in $\mathbb{N}[q]$.
\end{theorem}

We show that for the case when $B=S_{\lambda}$, our $q$-hit numbers
are equal to the Garsia--Remmel $q$-hit numbers, up to a power of $q$.

\begin{proposition}
\label{prop:q-hit and gr}
For any Ferrers board $S_{\lambda} \subset [n]\times [n]$ and $i=0,1,\ldots,n$, we have
\[
\nHit_i(S_{\lambda},q) = q^{\binom{n}{2}} H^{\GR}_i(S_{\lambda},q).
\]
\end{proposition}

\begin{proof}
Since $S_{\lambda}$ has the NE property,  Theorem~\ref{thm:polyqNE} allows us
to replace $M_i(S_{\lambda}, q)$ in \eqref{eq:intro-hit}
(the defining equation for ${P}(S_{\lambda}, q, t)$) 
with $q^{|\lambda|-i} R_i^{\NE}(S_{\lambda},q^{-1})$ to obtain
\begin{align*}
{P}(S_{\lambda}, q, t) & = q^{\binom{n}{2} + |\lambda|}\sum_{i=0}^n
R^{\NE}_{i}(S_{\lambda},q^{-1})q^{-i} [n-i]!_q \prod_{k=0}^{i-1} (tq^{-k}-1) \\
& = q^{\binom{n}{2} + |\lambda|}\sum_{i=0}^n
R^{\NE}_{i}(S_{\lambda},q^{-1}) q^{-i} \left( q^{\binom{n - i}{2}} \cdot [n-i]!_{q^{-1}} \right)\left(q^{ni - \binom{i}{2}} \cdot \prod_{k=0}^{i-1} (tq^{-n}-q^{-(n - k)})\right).
\end{align*}
By comparing the right side of this equation with that of
\eqref{HaglundRemmelversion} and rearranging powers of $q$, we see that
\[
P(S_{\lambda},q,t) = q^{2\binom{n}{2} + |\lambda|} \sum_{i=0}^n
\Hit_i^{\GR}(S_{\lambda},q^{-1}) q^{-ni} t^i.                                                                                
\]
Equating the coefficients of $t$ on both sides yields
\begin{equation} \label{eq:relhits1}
\nHit_{i}(S_{\lambda},q) = q^{2\binom{n}{2} + |\lambda| -in}H^{\GR}_i(S_{\lambda},q^{-1}).
\end{equation}
Haglund \cite[\S 5]{Haglund} and Dworkin \cite[Thm.~9.22]{Dworkin}
independently showed that the Garsia--Remmel $q$-hit numbers are symmetric, i.e., 
\begin{equation} \label{eq:GRhitsym}
q^{\binom{n}{2}+|\lambda|-in}H_{i}^{\GR}(S_{\lambda},q^{-1}) = H^{\GR}_i(S_{\lambda},q).
\end{equation}
Combining \eqref{eq:GRhitsym} and \eqref{eq:relhits1} gives the
desired expression.
\end{proof}

\begin{remark}
\label{rmk:statistics}
Garsia and Remmel proved Theorem~\ref{thm:qposHit} using recurrences for
the generating polynomial of their $q$-hit numbers (their analogue of
our ${P}(B, q, t)$) that preserve $q$-positivity.  
Later, Haglund and Dworkin gave (different) statistics $\stat_H$ and $\stat_{D}$ on permutations such that 
\[
H^{\GR}_i(S_{\lambda},q) = \sum_{\sigma} q^{\stat_H(\sigma)} =
\sum_{\sigma} q^{\stat_D(\sigma)},
\]
where the sum is over permutations $\sigma\in \Sym_n$ with $|\sigma \cap
S_{\lambda}|=i$. For the question of giving a combinatorial
interpretation to the numbers $\nHit_i(B,q)$, see Section~\ref{sec:hit-combinatorial-interpretation}.
\end{remark}

\begin{remark}
\label{rmk:rearrangement}
Dworkin also showed that his statistic $\stat_D(\cdot)$  is invariant
under permuting the columns of the Ferrers boards. In contrast, by
Proposition~\ref{prop:invariance}, for \emph{any} board $B\subseteq[m]
\times [n]$ the $q$-hit numbers $\nHit_i(B,q)$ 
are invariant under permuting rows
and columns of the board. 
\end{remark}

\subsection{A $q$-analogue of the probl\`eme des m\'enages} \label{sec:menage-examples}

In Example~\ref{ex:qderangements}, we gave a $q$-analogue of the
problem of derangements. Two other classical combinatorial
problems involve the boards
\[
B \defeq \{(1,1),\ldots,(n,n),
(1,2),(2,3),\ldots,(n-1,n)\} 
 \quad \text{and} \quad    B' \defeq B \cup \{(n,1)\},
\]
illustrated in Figure~\ref{fig:meange4by4board}.

\begin{figure}
\begin{center}
 \includegraphics{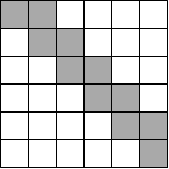}  \qquad \qquad \qquad   \includegraphics{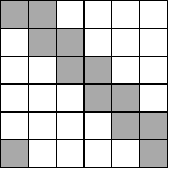}
\end{center}
\caption{The bidiagonal board $B$ and the \emph{m\'enage} board $B'$ in the case $n = 6$.}
\label{fig:meange4by4board}
\end{figure}
The rook numbers 
\[
r_i(B) = \binom{2n - i}{i} 
\qquad \textrm{and} \qquad 
r_i(B') = \frac{2n}{2n - i} \binom{2n - i}{i}
\] 
respectively count the number of ways of choosing $i$ points, no two consecutive, from 
a linear collection of $2n - 1$ points and a cyclic collection of $2n$ points. 
The hit numbers $h_0(B)$ and $h_0(B')$ respectively count permutations $w$ in
$\Sym_n$ such that $w_i \neq i, i+1$ and 
$w_i \not\equiv i,i+1 \pmod{n}$. These numbers have the formulas
\begin{equation} \label{eq:menage}
h_0(B)  = \sum_{i=0}^n (-1)^i \binom{2n-i}{i} (n-i)! 
\qquad \textrm{and} \qquad
h_0(B') = \sum_{i=0}^n (-1)^i\frac{2n}{2n-i} \binom{2n-i}{i} (n-i)!. 
\end{equation}
The rook theory of $B'$ is the famous {\em probl\`eme des m\'enages} 
\cite[Ex.~2.3.3]{EC1}.  

According to Haglund (private communication), he and Rota considered 
the problem of finding a $q$-analogue of the zero hit number $h_0(B')$.  
Here, we give formulas for both $H_0(B,q) = M_n(\overline{B},q)$ and
$H_0(B',q) = M_n(\overline{B'},q)$, $q$-analogues of $h_0(B)$ and
$h_0(B')$, respectively.  We begin by computing the $q$-rook numbers of $B$ and $B'$.

\begin{lemma}
\label{lem:rook numbers reduced menage complement}
The number of $n\times n$ matrices of rank $i$ with support on $B$ is
\[
\mat_i(B,q) = (q-1)^i \left(\binom{2n-1-i}{i-1} q^{i-1} +
  \binom{2n-1-i}{i} q^{i} \right).
\]
\end{lemma}
\begin{proof}
Let $B^{\star}$ be the reflection of $B$ through a horizontal axis, that is, $B^{\star} = \{(n,1), (n-1, 2), \ldots, \\ (1,n),(n,2),(n-1,3),\ldots,(2,n)\}$.  By inspection, $B^{\star}$ has the NE property (in a vacuous way).  Since $q$-rook numbers are invariant under permutations of the board, we have $\mat_i(B, q) = \mat_i(B^{\star}, q)$.  Since $B^{\star}$ has the NE property, we have by Theorem~\ref{thm:polyqNE} that $\mat_i(B^{\star}, q) = (q - 1)^i q^{2n - 1 - i} R^{\NE}_i(B^{\star}, q^{-1})$.  Thus, it suffices to compute the Garsia--Remmel $q$-rook number of $B^\star$.

Of the $\binom{2n - i}{i}$ placements of $i$ rooks on $B^\star$, exactly $\binom{2n - 1 - i}{i - 1}$ include a rook in position $(1, n)$, while the remaining $\binom{2n - 1 - i}{i}$ leave this cell empty.  By the definition of NE inversion given at the beginning of Section~\ref{sec:NE property}, each rook placed on $B^\star$ kills two potential NE inversions, except if the rook is placed on $(1, n)$, in which case it kills only one.  Thus, a rook placement of $i$ rooks on $B^\star$ has $2n - 2i$ NE inversions if it includes $(1, n)$, and $2n - 2i - 1$ NE inversions if not.  Combining these statements with the preceding paragraph gives the desired result.
\end{proof}

Unlike $B$, the board $B'$ does not have a rearrangement with the NE property for $n \geq 3$.  Thus, we use a different technique to compute the $q$-rook numbers for this board.

\begin{lemma}
\label{lem:rook numbers menage complement}
The number of $n\times n$ matrices of rank $i$ with support on $B'$ is
\[
\mat_i(B', q) = (q-1)^i \left( q^i \frac{2n}{2n-i} \binom{2n-i}{i} + G_{i, n}(q)\right),
\]
where $G_{n, n}(q) \defeq -q(q - 1)^{n - 1}$, $G_{n - 1, n}(q) \defeq (q - 1)^n$, and $G_{i, n}(q) \defeq 0$ if $i < n - 1$.
\end{lemma}
\begin{proof}
Denote $B_n \defeq B$ and $B'_n \defeq B'$ to make the dependence on the size $n$ explicit.  We proceed by induction.  When $n = 2$, we have that $B'_2 = [2] \times [2]$ is the entire square, and it is easy to check that the $q$-rook numbers in this case are $1 = \frac{4}{4 - 0} \binom{4}{0}$, $(q - 1)(q + 1)^2 = (q - 1)q \cdot \frac{4}{3} \binom{3}{1} + (q - 1)^3$, and $(q - 1)^2 q(q + 1) = (q - 1)^2q^2 \cdot \frac{4}{2} \binom{2}{2} - q(q - 1)^3$.  

If $n > 2$, we use recurrences from \cite[\S 3.2.3]{KleinLM}; for completeness, we sketch the argument here.  The set of matrices of rank $i$ with support on $B'_n$ may be written as a disjoint union of three pieces: those with $(n, 1)$ entry equal to $0$, those, with $(n, 1)$ entry nonzero but $(n, n)$ entry equal to $0$, and those with both $(n, 1)$ and $(n, n)$ entry nonzero.  The first of these subsets is the set of matrices of rank $i$ with support on $B_n$ and so has size $\mat_i(B_n, q)$.  By using Gaussian elimination to kill the $(1, 1)$-entry, we have that a matrix with the correct support belongs to the second subset if and only if its $[n - 1] \times [2, n]$-submatrix is of rank $i - 1$.  In this case, the submatrix is the transpose of a matrix with support on $B_{n - 1}$, and so the second subset has size $(q - 1)q \cdot \mat_{i - 1}(B_{n - 1}, q)$.  Finally, a matrix with the correct support belongs to the third subset if and only if, after using Gaussian elimination to kill the $(1, 1)$-entry, the $[n - 1] \times [2, n]$-submatrix is of rank $i - 1$ with support on the transpose of $B'_{n - 1}$.  Thus, the third subset has size $(q - 1)^2 \cdot \mat_{i - 1}(B'_{n - 1}, q)$.  Using Lemma~\ref{lem:rook numbers reduced menage complement} and the inductive hypothesis, it follows that
\begin{align*}
\frac{\mat_i(B'_n, q)}{(q - 1)^i} & = \mat_i(B_n, q) + (q - 1)q \cdot \mat_{i - 1}(B_{n - 1}, q) + (q - 1)^2 \cdot \mat_{i - 1}(B'_{n - 1}, q) \\
& = q^i \frac{2n}{2n-i} \binom{2n-i}{i} + G_{i, n}(q),
\end{align*}
as desired.
\end{proof}

With these computations in hand, it is straightforward to give a $q$-analogue of the \emph{m\'{e}nage} problem.

\begin{theorem}[$q$-analogue of \emph{m\'{e}nages}] \label{thm:qmenage}
The number of invertible $n\times n$
matrices over the finite field with $q$ elements with zeros on the main and upper diagonal is
\[
\mat_n(\overline{B}, q)  =  (q-1)^n q^{\binom{n}{2} - 2n} \sum_{i=0}^n
(-1)^i q^i\left(
  q\binom{2n-1-i}{i}+\binom{2n-1-i}{i-1} \right) [n-i]!_q.
\]
For $n \geq 2$, the number of invertible $n\times n$ matrices over the finite field with $q$ elements with zeros on the
main and upper diagonal and on the entry $(n,1)$ is
\[
\mat_n(\overline{B'}, q)  =  (q-1)^n q^{\binom{n}{2} - 2n}
\left( (-1)^{n-1} (q-1)^{n-1}(2q-1) + \sum_{i=0}^{n} (-1)^i q^i \frac{2n}{2n-i} \binom{2n-i}{i}[n-i]!_q \right).
\]
\end{theorem}

\begin{proof}
Applying \eqref{eq:fullrankrect} to Lemmas~\ref{lem:rook numbers reduced menage complement} and~\ref{lem:rook numbers menage complement} in the case $m=n$ and rearranging powers of $q$ and $q - 1$ gives the result.
\end{proof}

The preceding result should be compared with \eqref{eq:menage}.
Observe that in the last formula, the anomalous term $(-1)^{n - 1}(q - 1)^{n - 1}(2q - 1)$ vanishes modulo $q - 1$, in agreement with \eqref{eq:Mq-analogue} and Proposition~\ref{prop:qhitqanalogue}.  As with the $q$-analogue of derangements (Example~\ref{ex:qderangements}), these polynomials in general do not have positive coefficients

\section{Deletion-contraction for complements of boards 
          with the NE  property}
\label{sec:deletion-contraction}

In this section, we give deletion-contraction recurrence relations to compute the matrix count $M_r(B,q)$ and $q$-hit polynomial ${P}(B,q,t)$ when the board $B$ is the complement of a shape with the NE property.  Given a board $B$, say that an element $(i, j) \in B$ is a \emph{SW corner} if there is no other element $(i', j') \in B$ such that $i' \geq i$ and $j' \leq j$.

\subsection{General relations}
Given a board $B \subseteq [m] \times [n]$ with the NE property, let $\square$ be a SW corner of $B$. Denote by $B\setminus \square$ the board obtained by deleting $\square$
from $B$, and denote by $B/\square$ the board obtained by deleting 
the entire row and column of $\square$.  For purposes of taking complements, we think of $B/\square$ as living inside the smaller rectangle $[m - 1] \times [n - 1]$.  The following result of Dworkin gives a deletion-contraction relation for the Garsia--Remmel $q$-rook numbers.

\begin{proposition}[{\cite[Thm.~6.10]{Dworkin}}] \label{prop:delcon}
For any board $B\subseteq[m]\times [n]$ and SW corner $\square$ of $B$,  
\[
R_r^{\NE}(B,q) = q \cdot R_r^{\NE}(B\setminus \square, q) + R^{\NE}_{r-1}(B/\square,q).
\]
\end{proposition}

\begin{corollary} \label{cor:delcontract}
For any board $B\subseteq[m]\times [n]$ with the NE property and any SW corner $\square$ of $B$, 
\begin{equation}
\label{eq:matdelcon} 
M_r(B,q) = M_r(B\setminus \square,q) + q^{|B|-|B/\square|-1} \cdot M_{r-1}(B/\square,q).
\end{equation}
\end{corollary}
\begin{proof}
If $B$ has the NE property and $\square$ is a SW corner then both
$B\setminus \square$ and $B/\square$ have the NE property. So using
Theorem~\ref{thm:polyqNE} we can rewrite the deletion-contraction in
Proposition~\ref{prop:delcon} in terms of $M_r$.
\end{proof}

The next result shows how to pass this deletion-contraction relation through the complement identity \eqref{eq:MW} to produce a recurrence counting matrices on complements of boards with the NE property.  In the proof, it is necessary to consider simultaneously $q$-\K polynomials defined with different parameters $m, n$; thus, for the duration of this section we introduce the notation 
\begin{equation}
\label{eq:K with m, n}
K^{m,n}_r(i)\defeq \sum_{s} (-1)^{r-s} q^{ns+\binom{r-s}{2}} \qbin{m-s}{r-s}{q}\qbin{m-i}{s}{q}
\end{equation}
for the polynomial previously denoted $K_r(i)$.  By applying the $q$-Pascal recurrence
\[
\qbin{k}{\ell}{q} = \qbin{k - 1}{\ell - 1}{q} + q^\ell \cdot \qbin{k - 1}{\ell}{q}
\]
to the first $q$-binomial coefficient appearing in \eqref{eq:K with m, n}, it is not hard to show that
\begin{equation}
\label{eq:K recurrence}
K^{m, n}_r(j + 1) = q^r \cdot K^{m-1, n-1}_r(j) - q^{r-1} \cdot K^{m-1, n-1}_{r-1}(j).
\end{equation}

\begin{corollary} \label{cor:delcontractcomp}
For any board $B\subseteq[m]\times [n]$ with the NE property and any SW corner $\square$ of $B$,
\[
q\cdot M_r(\overline{B}, q) =  M_r(\overline{B \setminus \square}, q) + q^r(q-1)\cdot M_r(\overline{B/\square}, q) - q^{r-1}\cdot M_{r-1}(\overline{B/\square}, q).
\]
\end{corollary}

\begin{proof}
Applying the complement identity \eqref{eq:MW} to the matrix count $\mat_r(\overline{B},q)$ gives
\[
\mat_r(\overline{B},q) 
= 
\frac{1}{q^{|B|}} \sum_{i=0}^m K^{m,n}_r(i)\cdot \mat_i({B},q)
.
\]
Choose a SW corner $\square$ of $B$.  Then applying \eqref{eq:matdelcon} to $\mat_i(B,q) = (q - 1)^i \cdot M_i(B,q)$ gives
\[
\mat_r(\overline{B},q) =
q^{-|B|} \sum_{i=0}^m K^{m,n}_r(i)\cdot \mat_i(B\setminus \square,q)
+
\frac{q-1}{q^{|B/\square| + 1}} \sum_{i=0}^m K^{m,n}_r(i)\cdot \mat_{i-1}({B /\square},q).
\]
The first summand simplifies by \eqref{eq:MW} (using the board 
$B\setminus \square \subseteq [m]\times [n]$) to $\frac{1}{q}\mat_r(\overline{B \setminus \square},q)$. 
For the second summand, we have by \eqref{eq:K recurrence} that
\begin{align*}
 \sum_{i=0}^m K^{m,n}_r(i) \mat_{i-1}({B /\square},q) 
& =  \sum_{j=0}^{m-1}  K^{m,n}_r(j+1) \mat_j(B/\square,q) \\
& = q^r \sum_{j=0}^{m-1} K^{m-1,n-1}_r(j) \mat_j({B /\square},q)
    - q^{r - 1}\sum_{j=0}^{m-1} K^{m-1,n-1}_{r-1}(j)  \mat_j({B /\square},q).
\end{align*}
Both of the sums in this last expression can again be transformed using 
\eqref{eq:MW}, the first using the board 
$B/\square \subseteq [m-1]\times [n-1]$ at rank $r$, 
and the second using the board 
$B/\square \subseteq [m-1]\times [n-1]$ and rank $r-1$.  This gives
\[
 \sum_{i=0}^m K^{m,n}_r(i) \cdot \mat_{i-1}({B /\square},q) =
q^{r + |B / \square|} \cdot \mat_r(\overline{B /\square},q)
-
q^{r - 1 + |B / \square|} \cdot \mat_{r-1}(\overline{B /\square},q).
\]
Putting everything together, we get
\[
\mat_r(\overline{S},q) 
 = \frac{1}{q} \cdot \mat_r(\overline{B \setminus \square},q) + (q-1)\left(q^{r-1} \cdot \mat_r(\overline{B/\square},q) - q^{r-2} \cdot \mat_{r-1}(\overline{S/\square},q)\right).
\]
Multiplying by $q (q - 1)^{-r}$ on both sides gives the desired result.
\end{proof}

We can also transform these deletion-contraction relations 
in terms of the $q$-hit polynomial.

\begin{corollary} \label{cor:delcon-q-hit-poly}
For any board $B \subseteq[m]\times [n]$ with the NE property and any SW corner $\square$ of $B$,
\[
{P}(B,q,t) = 
{P}(B\setminus \square,q,t) + 
q^{m+|B|-|B/\square|-2} (t-1) \cdot {P}(B / \square,q, q^{-1}t)
\]
and
\[
q\cdot {P}(\overline{B},q,t) = 
{P}(\overline{B\setminus \square},q,t)  -
q^{m-1}(t - q^n) \cdot 
{P}(\overline{B/\square},q,t).
\]
\end{corollary}

\begin{proof}
To show the first relation, we apply \eqref{eq:matdelcon} to the definition \eqref{definition of hits inside rectangle} to get
\begin{multline*}
{P}(B,q, t) =
q^{\binom{m}{2}}\sum_{r=0}^m M_r(B\setminus\square,q) \frac{[n - r]!_q}{[n - m]!_q} 
  (-1)^r (t; q^{-1})_r + {}\\
{} + q^{\binom{m}{2} +|B|- |B/\square| - 1}
  \sum_{r=0}^{m}M_{r-1}(B/\square,q) \frac{[n - r]!_q}{[n - m]!_q} 
  (-1)^r (t; q^{-1})_r.
\end{multline*}
The first sum equals ${P}(B\setminus \square, q, t)$. 
For the second sum,
changing the index of summation to $i=r-1$ and factoring out the
term $t-1$ produces
\begin{multline*}
q^{\binom{m}{2}  + |B|- |B/\square| - 1}(t - 1) \sum_{i=0}^{m - 1}M_i(B/\square,q) \frac{[(n - 1) - i]!_q}{[(n - 1) - (m - 1)]!_q} (-1)^i (tq^{-1};q^{-1})_i = {}\\
q^{|B|-|B/\square|+m - 2} (t - 1){P}(B / \square,q,q^{-1}t).
\end{multline*}
Substituting back in gives the desired relation.

To show the second relation, we use the reciprocity formula 
\eqref{eq:hitreciprocity2}
to rewrite the first relation in terms of ${P}(\overline{B},q,t)$:
\begin{multline*}
q^{-|\overline{B}|} t^m \cdot {P}(\overline{B},q,q^{n}t^{-1}) = {} \\
     q^{-|\overline{B\setminus\square}|} t^m \cdot 
     {P}(\overline{B\setminus\square},q,q^{n}t^{-1})
                                                    +
   q^{|B|-|B/\square|+m-2-|\overline{B/\square}|}(t-1)t^{m - 1} \cdot {P}(\overline{B/\square},q,q^{n}t^{-1}).
\end{multline*}
Dividing both sides by $q^{- |\overline{B}|-1}t^m$ and substituting
$t\mapsto q^{n}t^{-1}$ gives the result.
\end{proof}

\subsection{Permutation diagrams and deletion-contraction}

In this section, we study deletion-contraction on diagrams $I_w
\subseteq [n]\times [n]$ of permutations $w$ of size $n$. We show that
the deletion and contraction boards from
Corollaries~\ref{cor:delcontract},~\ref{cor:delcontractcomp} are 
actually diagrams of related permutations.

For $w \in \Sym_n$, let $\square=(i,w_j)$ be a SW corner of
$I_w \subseteq [n]\times [n]$.
By definition, there is no other element $(i',w_{j'})$ in $I_w$ (so $i'<j'$, $w_{i'}<w_{j'}$) with $i \leq i'$ and $w_{j'} \leq w_j$.  In particular, in this case none of the entries $w_{i + 1}, \ldots, w_{j - 1}$ have values between $w_i$ and $w_j$, and so the permutation
\[
w\cdot (i,j) = w_1\cdots w_{i-1} \underline{w_j} w_{i+1}\cdots
w_{j-1} \underline{w_i} w_{j+1}\cdots w_n
\]
has exactly one fewer coinversion than $w$.  Next, we show that in fact, $I_{w \cdot (i, j)}$ arises by deleting a single cell from $I_w$.

\begin{proposition}[deletion] \label{prop:delw}
For any $w\in \Sym_n$ and any SW corner $\square=(i,w_j)$ of $I_w$, we have
$I_w \setminus \square = I_{w \cdot (i, j)}$.
\end{proposition}
\begin{proof}
Abbreviate $w' \defeq w \cdot (i, j)$.
Consider a box $(a, w_b)$ of $I_w$ corresponding to a coinversion between entries $(a, w_a)$ and $(b, w_b)$ of $w$ (so necessarily $a < b$ and $w_a < w_b$).  If $a = i$ and $b = j$ then obviously this box is absent in $I_{w'}$; we show that all other boxes in $I_w$ are in $I_{w'}$.

First, if $\{a, b\}$ is disjoint from $\{i, j\}$, then $(a, w_a) = (a, w'_a)$ and $(b, w_b) = (b, w'_b)$ are entries of $w'$ and so $(a, w_b)$ belongs to $I_{w'}$.

Second, suppose that $a = j$, and so $b > j$ and $w_b > w_a = w_j > w_i$.  Then the entries $(a, w'_a) = (j, w_i)$ and $(b, w'_b) = (b, w_b)$ of $w'$ form a coinversion and so $(a, w_b)$ belongs to $I_{w'}$.

Third, suppose that $a = i$ (and so $b > i$).  Since $\square$ is a SW corner of $I_w$, we must have $w_b > w_j$.  Then the entries $(i, w'_i) = (i, w_j)$ and $(b, w'_b) = (b, w_b)$ of $w'$ form a coinversion and so $(a, w_b)$ belongs to $I_{w'}$.

Fourth, the cases $b = i$ and $b = j$ are symmetric with the last two cases after reflecting everything across the main antidiagonal.

Finally, since $I_w$ has strictly more coinversions than $I_{w'}$, it follows that these are all the elements of $I_{w'}$, as desired.
\end{proof}

Next, we show that contractions of permutation diagrams are also permutation diagrams.
\begin{definition}
\label{def:contr}
Given a permutation $w \in \Sym_n$ whose diagram $I_w$ has a SW corner in position $(i, w_j)$, 
let $v \in \Sym_{n-1}$ be the permutation order-isomorphic to
\[
w_1\cdots w_{i-1} w_{i+1} \cdots w_{j-1} \underline{w_i} w_{j+1}\cdots w_n.
\]
\end{definition}
For example, the diagram of the permutation $w = 139547628 \in \Sym_9$ has a SW corner in position $(5, w_7) = (5, 6)$ (see Figure~\ref{fig:contraction}), and $v = 13856427 \in \Sym_8$ is the permutation order-isomorphic to $13957428$.

\begin{figure}
\begin{center}
\includegraphics[scale=1.2]{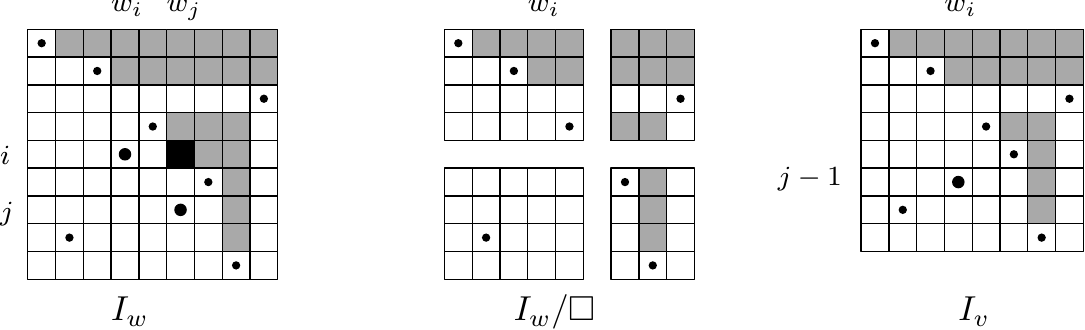}
\end{center}
\caption{Left: The diagram of the permutation $w = 139547628$, with SW corner $\square = (5, 6)$ involving the entries $(i, w_i) = (5, 4)$ and $(j, w_j) = (7, 6)$.  Center: the contracted diagram.  Right: the diagram of the permutation $v = 13856427$.}
\label{fig:contraction}
\end{figure}

\begin{proposition}[contraction] \label{prop:contrw}
For any $w\in \Sym_n$ and any SW corner $\square=(i,w_j)$ of $I_w$, we have 
$I_w / \square = I_{v}$, where $v$ is the permutation defined in Definition~\ref{def:contr}.
\end{proposition}

\begin{proof}
Let $r$ and $c$ be the order-preserving bijections $r: [n] \setminus \{i\} \to [n - 1]$ and $c: [n] \setminus \{w_j\} \to [n - 1]$, so that $f \defeq r \times c$ is the natural bijection between cells of $[n] \times [n]$ not in the row or column of $\square$ and cells in $[n - 1] \times [n - 1]$.  We seek to show that $f$ restricts to a bijection between the relevant cells of $I_w$ and those of $I_v$.  Observe that by the definition of $v$, we have $v_{r(a)} = c(w_a)$ for all $a \neq j$, while $v_{r(j)} = c(w_i)$.

Consider a box $(a,w_b)$ of $I_w/\square$ corresponding to a coinversion
between entries $(a,w_a)$ and $(b,w_b)$ of
$w$ (so $a<b$, $w_a<w_b$, $a \neq i$, $b \neq j$).  We have three cases.

First, assume $a \neq j$ and $b \neq i$.  Since $r$ and $c$ are order-preserving, we have $r(a) < r(b)$ and $v_{r(a)} = c(w_a) < c(w_b) = v_{r(b)}$.  Therefore $f(a, w_b) = (r(a), v_{r(b)})$ belongs to $I_v$.  Conversely, if $(r(a), v_{r(b)}) \in I_v$ is such that $a \neq j$ and $b \neq i$ then $f^{-1}(r(a), v_{r(b)}) = (a, w_b) \in I_w$.

Second, suppose $b = i$.  Then $a < i < j$ and $w_a < w_i$.  Thus $r(a) < r(j)$ and $v_{r(a)} = c(w_a) < c(w_i) = v_{r(j)}$, and therefore $f(a, w_i) = (r(a), v_{r(j)})$ belongs to $I_v$.  Conversely, suppose $(r(a), v_{r(j)}) \in I_v$, so that $r(a) < r(j)$ and $c(w_a) = v_{r(a)} < v_{r(j)} = c(w_i)$.  Since $\square = (i, w_j)$ is a SW corner in $I_w$, all of the values $w_{i + 1}, \ldots, w_{j - 1}$ must be larger that $w_j$, and so also larger than $w_i$.  Therefore, $a$ is not equal to any of $i + 1, \ldots, j - 1$, so $a < i$.  Thus $f^{-1}(r(a), v_{r(j)}) = (a, w_i)$ belongs to $I_w$.

Finally, suppose $a = j$. Then $b > j$ and $w_b > w_j > w_i$.  Thus $r(b) > r(j)$ and $v_{r(b)} = c(w_b) > c(w_i) = v_{r(j)}$, and therefore $f(j, w_b) = (r(j), v_{r(b)})$ belongs to $I_v$.  Conversely, suppose $(r(j), v_{r(b)}) \in I_v$, so that $r(j) < r(b)$ and $c(w_i) = v_{r(j)} < v_{r(b)} = c(w_b)$.  Since $\square = (i, w_j)$ is a SW corner in $I_w$, all of the values $w^{-1}_{w_i + 1}$, \ldots, $w^{-1}_{w_j - 1}$ must be smaller than $i$, and so also smaller than $j$.  Therefore, $b$ is not equal to any of these values, so $b > j$.  Thus $f^{-1}(r(j), v_{r(b)}) = (i, w_b)$ belongs to $I_w$.

The three cases cover every coinversion of $w$ not in row $i$ or column $w_j$ and every coinversion of $v$, and so $f$ is a bijection between these two sets.  The result follows immediately.
\end{proof}

\begin{corollary}
For any permutation $w\in \Sym_n$, let $\square=(i,w_j)$ be any SW corner of $I_w$ and let $v$ be the permutation defined in Definition~\ref{def:contr}.  For any $1\leq r \leq n$, we have 
\begin{equation} \label{eq:delconMw}
M_r(I_w,q) =  M_r(I_{w \cdot (i, j)},q) + q^{n-2-\ell(w)+\ell(v)} M_{r-1}(I_v,q).
\end{equation}
\end{corollary}

\begin{proof}
Apply \eqref{eq:matdelcon} in the case $m=n$ and
$B=I_w$, using Propositions~\ref{prop:delw},~\ref{prop:contrw} to
express $I_w\setminus \square$ and $I_w /\square$ as $I_{w \cdot (i, j)}$ and $I_v$ respectively. Since $|I_w|=\binom{n}{2}-\ell(w)$, we have $|I_w|-|I_v|-1=n-2-\ell(w)+\ell(v)$.
\end{proof}

\begin{corollary} \label{cor:delconw}
For any permutation $w\in \Sym_n$, let $\square=(i,w_j)$ be any SW corner of $I_w$ and let $v$ be the permutation defined in Definition~\ref{def:contr}.  For any $1\leq r \leq n$, we have 
\[
q\cdot M_r(\overline{I_w},q) =  M_r(\overline{I_{w \cdot (i, j)}},q) +q^r(q-1) \cdot M_r(\overline{I_v},q) - q^{r-1}\cdot M_{r-1}(\overline{I_v},q).
\]
\end{corollary}

\begin{proof}
Apply Corollary~\ref{cor:delcontractcomp} in the case $m=n$ and
$B=I_w$, using Propositions~\ref{prop:delw},~\ref{prop:contrw} to
express $I_w\setminus \square$ and $I_w /\square$ as $I_{w \cdot (i, j)}$ and $I_v$ respectively. 
\end{proof}

The preceding result is particularly nice in the full-rank case.

\begin{corollary}
\label{cor:full rank}
For any permutation $w\in \Sym_n$, let $\square=(i,w_j)$ be any SW corner of $I_w$ and let $v$ be the permutation defined in Definition~\ref{def:contr}.  We have 
\[
q \cdot M_n(\overline{I_w},q) =  M_n(\overline{I_{w \cdot (i, j)}},n) - q^{n-1} \cdot M_{n-1}(\overline{I_v},q).
\]
\end{corollary}

Finally, we can rewrite these relations in terms of the $q$-hit polynomial ${P}$.

\begin{corollary}
For any permutation $w\in \Sym_n$, let $\square=(i,w_j)$ be any SW corner of $I_w$ and let $v$ be the permutation defined in Definition~\ref{def:contr}.  We have
\[
{P}(I_w,q,t) = {P}(I_{w \cdot (i, j)},q,t) + q^{2n-3+\ell(v)-\ell(w)}(t-1) \cdot {P}(I_v,q,q^{-1}t)
\]
and 
\[
q\cdot {P}(\overline{I_w},q,t) = {P}(\overline{I_{w\cdot (i, j)}},q,t) -
q^{n-1}(t-q^n) \cdot {P}(\overline{I_v},q,t).
\]
\end{corollary}
\begin{proof}
Combine Corollary~\ref{cor:delcon-q-hit-poly}
for $m=n$, $B=I_w$ with Propositions~\ref{prop:delw} and~\ref{prop:contrw}.
\end{proof}

\subsection{Failures of positivity}
\label{sec:failures of positivity}

There are no permutations in $w \in \Sym_n$ for $n < 9$ for which any
coefficient of $M_r(\overline{I_w},q)$ is negative for any $r$.  However, for $n \geq 9$, there are counterexamples to the positivity aspect of Conjecture~\ref{rothe conjecture}.

\begin{figure}
\begin{tabular}{cp{1in}c}
\includegraphics{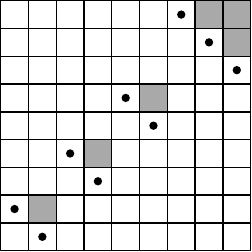}
&&
\includegraphics{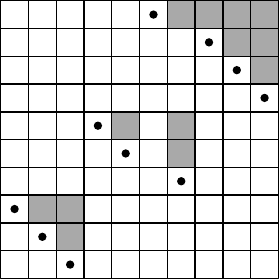}
\\
(a)
&&
(b)
\end{tabular}
\caption{The diagrams of the two permutations $789563412 \in \Sym_9$ (left) and $6\,8\,9\,10\,4\,5\,7\,1\,2\,3 \in \Sym_{10}$ (right).}
\label{fig:failures of positivity}
\end{figure}

\begin{example}
Let $w=789563412 \in \Sym_9$; its diagram is shown in Figure~\ref{fig:failures of positivity}(a).  Applying Corollary~\ref{cor:delconw} (or the formula in \cite[Prop.~3.1]{KleinLM}) gives
\[
M_1(\overline{I_w},q) 
=
24q^{11}-4q^{10}+10q^9+9q^8+8q^7+7q^6+6q^5+5q^4+4q^3+3q^2+2q+1.
\]
\end{example}
There are three other permutations in $\Sym_9$ whose diagrams are 
trivial rearrangements of the previous example (namely, 
$895673412$, $896734512$, and $896745123$).  These are the only 
permutations in $\Sym_9$ for which $M_r(\overline{I_w}, q)$ has some negative 
coefficients, and they only have negative coefficients in rank $r = 1$.

\begin{example}
Let $w =6\,8\,9\,10\,4\,5\,7\,1\,2\,3 \in \Sym_{10}$; its diagram is shown in Figure~\ref{fig:failures of positivity}(b).  Applying Corollary~\ref{cor:full rank} gives 
\[
M_{10}(\overline{I_w},q)=q^{77} + 9q^{76} + 44q^{75} + \cdots + 2q^{48} - 8q^{47} - q^{46} + q^{45}.
\]
\end{example}
In total, there are $37$ permutations $w$ in $\Sym_{10}$ for which $M_{10}(\overline{I_w}, q)$ has negative coefficients and $303$ for which $M_1(\overline{I_w}, q)$ has negative coefficients, including $11$ permutations for which both polynomials have negative coefficients.  The coefficients of $M_r(\overline{I_w}, q)$ are nonnegative for all $w \in \Sym_{10}$ if $2 \leq r \leq 9$.

These calculations are also sufficient to disprove another reasonable conjecture: that positivity is a pattern property.  Indeed, the permutation $w = 5\,8\,9\,10\,6\,7\,3\,4\,1\,2 \in \Sym_{10}$ has all coefficients of $M_r(\overline{I_w}, q)$ nonnegative for all ranks $r$, but it contains the permutation $789563412$ as a pattern.

\section{Remarks and open problems}
\label{sec:final}

\subsection{Combinatorial intepretation of $q$-hit numbers} \label{sec:hit-combinatorial-interpretation}

The main problem raised by our work is to give a combinatorial interpretation to the $q$-hit number $\nHit_i(B, q)$.
\begin{question}
\label{open question:interpretation}
Is there a nice choice of a set $S$ ($ = S(i, B, q)$) such that the cardinality $|S|$ is equal to the $q$-hit number $\nHit_i(B, q)$?
\end{question}
Corollary~\ref{cor:hit partition}, expressing the number of full-rank matrices of a given shape as a sum of hit numbers (times an easy-to-understand factor), suggests that a nice description would be in terms of a partition of the set of full-rank $m \times n$ matrices.  

For the case of Ferrers boards $S_{\lambda}$ and $n=m$, Haglund \cite{Haglund}
gave an interpretation for the
Garsia--Remmel $q$-hit numbers in terms of matrices over
finite fields.  (Recall that in this case the Garsia--Remmel $q$-hit
numbers agree with our $q$-hit numbers by Proposition~\ref{prop:q-hit
  and gr}.) This interpretation roughly goes
like this: by Proposition~\ref{prop:hitasMs}, the $k$th $q$-hit number
can be written as
\begin{multline*}
(q-1)^n q^{-\binom{n}{2}} \nHit_k(S_{\lambda},q) = \sum_{i=k}^{n} \mat_q(i,S_{\lambda}) \cdot
(q^n- q^{i})(q^{n-1}-q^{i})\cdots (q^{i+1}-q^{i}) \times {} \\
{} \times \qbin{i}{k}{q} (-1)^{i+k} q^{\binom{k+1}{2} - i(n+k-i)}.
\end{multline*}
View the right side as counting the following {\em replacement procedure} \cite[\S 2]{Haglund}:
\begin{enumerate}[\quad 1.]
\item start with a matrix $A$ of rank $i$ and support in $S_{\lambda}$,
\item after doing row elimination on $A$ in a specified order, replace, with certain
  rules, the $n-i$ rows without pivots to obtain a matrix $A'$ of full rank, 
\item assign a certain signed weight to the matrix $A'$, and
\item do row elimination on $A'$ and record the number of pivots
  $j$ that do not belong to $B$.
\end{enumerate}
Haglund showed that weighted contribution of such matrices with $j
<n-k$ is zero whereas the contribution of those with $j=n-k$ is one.

Unfortunately, we were unable to extend this elimination procedure to other families
of boards. One possible (but so far unsuccessful) approach is described in the next remark.

\begin{remark}
Equation \eqref{eq:classic_rookhit_rel} has a natural double-counting proof that one might try to emulate to find an interpretation for $\nHit_i(B,q)$.  
Consider \eqref{definition of hits inside rectangle} in the case $m =
n$; rearranging powers of $q$ and $q-1$, we have
\[
(q-1)^n\sum_{i=0}^n \nHit_{i}(B,q) t^i  \,=\, \sum_{i=0}^n \mat_i(B,q)
\cdot \prod_{k=i}^{n-1}(q^n-q^k) \cdot \prod_{j = 0}^{i - 1} (t-q^j).
\]
When $t=q^N$, we can view the right side as counting triples $(A,\beta,\phi)$
where $A$ is a matrix over $\F_q$ with support in $B$, 
$\beta$ is an ordered relative basis for $\F_q^n$ over the rowspace of $A$, 
and $\phi$ is an injective linear map from the rowspace to $\F_q^N$. 
\end{remark}

We also offer a weak conjecture that is certainly a precondition for an affirmative answer to Question~\ref{open question:interpretation}.
\begin{conjecture}
\label{weak positivity conjecture}
Given any board $B \subseteq [m] \times [n]$, rank $r$, and prime power $q$, the $q$-hit number $\nHit_r(B, q)$ is nonnegative.
\end{conjecture}

%

\subsection{Polynomiality and positivity of $q$-hit numbers} 

Just as is the case for the $q$-rook number $M_r(B,q)$, the 
$q$-hit number $\nHit_i(B,q)$ need not be a polynomial in $q$. In
fact, by Proposition~\ref{prop:hitasMs}, for a fixed board $B$,
we have that all $M_r(B,q)$ are polynomial in $q$ if and only if
all $\nHit_i(B,q)$ are polynomial. 

\begin{example} 
\label{ex:nonpolyhit}
In \cite{Stembridge}, Stembridge found a set $F \subseteq[7]\times [7]$
such that $M_7(F,q)$ is not a polynomial in $q$. This set $F$ is the
the incidence matrix of the Fano plane (see Figure~\ref{fig:Fano}). Stembridge
found that 
\begin{multline*}
M_7(F,x+1) = \nHit_7(F,x+1)= \left( x+1 \right) ^{3} (
x^{11} + 17x^{10} + 135x^{9} + 650x^{8} + 2043x^{7} + 4236x^{6} + {} \\
5845x^{5} + 5386x^{4} + 3260x^{3} + 1236x^{2} + 264x + 24 - Z_{2}x^{6})
\end{multline*}
where $x\defeq q-1$ and $Z_2$ is zero or one depending on whether $q$ is even or
odd.
\end{example}
\begin{figure}
\includegraphics{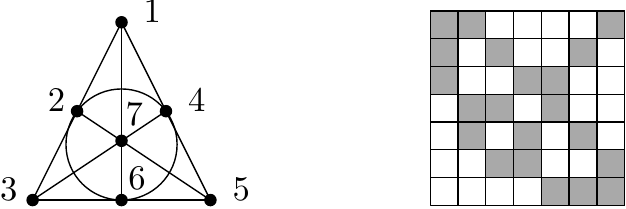}
\caption{The Fano plane and the associated board, which contains cell $(i, j)$ if and only if point $i$ lies on line $j$.}
\label{fig:Fano}
\end{figure}

Even when the $M_r(B,q)$ are polynomials with nonnegative coefficents, the
polynomial $\nHit_i(B,q)$ might have negative coefficients; see 
Example~\ref{ex:negpolyhit}(iii).  Thus it is natural to ask about
positivity if, e.g., we restrict to permutation
diagrams $I_w$ (where everything is polynomial by Corollary~\ref{cor:polynomiality-rothe}).

\begin{question}
For which permutations $w$ do the $q$-hit numbers $\nHit_i(\overline{I_w}, q)$ have positive coefficients?  
\end{question}
If $w \in \Sym_n$ avoids the permutation pattern $3412$ then
$\overline{I_w}$ is a rearrangement of a Ferrers board \cite[Prop.~2.2.7]{Manivel}.  
By Proposition~\ref{prop:q-hit and gr} and Remark~\ref{rmk:rearrangement}, for such
permutations $w$, the $q$-hit numbers $\nHit_i(\overline{I_w},q)$ equal
the Garsia--Remmel $q$-hit number $H^{\GR}_i(S_{\lambda},q)$ of the
associated Ferrers board, and so by
Theorem~\ref{thm:qposHit}, $\nHit_i(\overline{I_w}, q)$ is a polynomial
with positive coefficients.  For $n \leq 9$, these are the only such
permutations; it could be interesting to prove that this is true for
all $n$.

\begin{example}
For $w=3412$ we have that
\[
\nHit_0=\nHit_1=0, \quad
\nHit_2=q^{11}(q+1), \quad
\nHit_3=q^7(2q^4+4q^3+3q^2-1), \quad
\nHit_4=q^6(q^4+3q^3+5q^2+4q+1).
\]
\end{example}

Since $\nHit_r(\overline{I_w}, q)$ is a polynomial in $q$, one may also make a strengthened version of Conjecture~\ref{weak positivity conjecture} in this case.
\begin{question}
Is it true for every permutation $w$ and every rank $r$ that the polynomial $\nHit_r(\overline{I_w}, x + 1)$ has positive coefficients in the variable $x$?
\end{question}
The answer is affirmative for $n \leq 8$. The corresponding question for the matrix counts $M_r(\overline{I_w}, q)$ is posed in \cite[Rmk.~3.4]{KleinLM}.

\subsection{Positivity for $123$-avoiding permutations}
In \cite{LLMPSZ}, the motivating example was the board $B = ([n]
\times [n]) \setminus \{(1, 1), \ldots, (n, n)\}$; an elegant alternating formula
was given for the matrix count $M_n(B,q)$ (see Example~\ref{ex:qderangements}).  
This formula is not
positive.  Let $v \in \Sym_{2n}$ be the permutation $v \defeq (2n - 1)(2n)(2n -
3)(2n - 2) \cdots 563412$.  The diagram $I_v$ consists of $n$ boxes on
the diagonal, and by Corollary~\ref{cor:fullrankrect} one can write
down a similar alternating sum for the associated matrix count:
\[
M_{2n}(\overline{I_v},q) = q^{2n(n - 1)} \sum_{i = 0}^n (-1)^i \binom{n}{i} [2n - i]!_q.
\]
It is easy to check on a computer that for $n\leq 40$ these polynomials have nonnegative coefficients.

\begin{conjecture} \label{conj:poscompdiag}
For $v = (2n - 1)(2n)(2n - 3)(2n - 2) \cdots 3412$ 
we have that $M_{2n}(\overline{I_v},q)$ is in $\mathbb{N}[q]$.
\end{conjecture}

The diagonal board above is an example of a skew Ferrers board. Any
skew Ferrers board can be obtained (with our conventions for $I_w$) 
from the diagram of a $123$-avoiding permutation \cite{BJS}. 
This family, to which $v$ belongs, contains $\frac{1}{n+1}\binom{2n}{n}$
permutations in $\Sym_n$.  Calculations for $n\leq 14$ suggest that for every such
permutation, $M_n(w,q)$ has nonnegative coefficients. This suggests
the following strengthening of Conjecture~\ref{conj:poscompdiag}.

\begin{conjecture}
For every $123$-avoiding permutation $w$ in $\Sym_{n}$ we have that
$M_n(\overline{I_w},q)$ is in $\mathbb{N}[q]$.
\end{conjecture}

\bibliographystyle{alpha}
\bibliography{Remmel}

\end{document}